\documentclass[reqno]{amsart}
\usepackage[utf8]{inputenc}
\usepackage{amscd,amsfonts,amsmath,amssymb,amsthm,bbm,dsfont,centernot,mathtools}
\usepackage{color,mathrsfs,stackrel}
\usepackage[margin=1in]{geometry}
\usepackage[colorlinks=true,linkcolor=blue,citecolor=red]{hyperref}
\usepackage{yhmath}
\usepackage{fancyhdr}
\usepackage{xy}
\usepackage{eso-pic}
\usepackage{xcolor}
\usepackage{extarrows}
\usepackage{relsize}
\usepackage{caption}
\usepackage{eurosym}
\usepackage{chngcntr}

\newtheorem{corollary}{Corollary}[section]
\newtheorem{lemma}[corollary]{Lemma}

\newtheorem{theorem}[corollary]{Theorem}

\theoremstyle{definition}
\newtheorem{definition}[corollary]{Definition}
\newtheorem{remark}[corollary]{Remark}

\newtheorem*{acknowledgements}{\sc Acknowledgements}

\numberwithin{equation}{section}

\allowdisplaybreaks

\newcommand{\norm}[1]{\lVert #1 \rVert}
\newcommand{\spr}[2]{\langle #1, #2\rangle}

\def\Xint#1{\mathchoice
{\XXint\displaystyle\textstyle{#1}}%
{\XXint\textstyle\scriptstyle{#1}}%
{\XXint\scriptstyle\scriptscriptstyle{#1}}%
{\XXint\scriptscriptstyle\scriptscriptstyle{#1}}%
\!\int}
\def\XXint#1#2#3{{\setbox0=\hbox{$#1{#2#3}{\int}$ }
\vcenter{\hbox{$#2#3$ }}\kern-.6\wd0}}
\def\dashint{\Xint-}
\def\mint{\Xint{\rotatebox[origin][30]{$-$}}}

\def \Lip {\mathop {\rm Lip}\nolimits}

\def \div {\mathop {\rm div}\nolimits}

\def \dist {\mathop {\rm dist}\nolimits}

\def \de {\mathrm{d}}
\def \e {\epsilon}
\def \R {\mathbb R}

\def \E {E}

\begin{document}

\title{A dynamic model for viscoelastic materials with prescribed growing cracks}
\author[M. Caponi]{Maicol Caponi}
\address[Maicol Caponi]{SISSA, via Bonomea 265, 34136 Trieste, Italy}
\email{mcaponi@sissa.it}
\author[F. Sapio]{Francesco Sapio}
\address[Francesco Sapio]{SISSA, via Bonomea 265, 34136 Trieste, Italy}
\email{fsapio@sissa.it}
\thanks{Preprint SISSA 12/2019/MATE}

\begin{abstract}
In this paper, we prove the existence of solutions for a class of viscoelastic dynamic systems on time--dependent cracked domains, with possibly degenerate viscosity coefficients. Under stronger regularity assumptions we also show a uniqueness result. Finally, we exhibit an example where the energy--dissipation balance is not satisfied, showing there is an additional dissipation due to the crack growth.
\end{abstract}

\maketitle


\medskip


\section{Introduction}

In the theory of Dynamic Fracture, the deformation of an elastic material evolves according to the elastodynamics system, while the evolution of the crack follows Griffith's dynamic criterion, see~\cite{Mott}. This principle, originally formulated in~\cite{Grif} for the quasi--static setting, states that there is an exact balance between the energy released during the evolution and the energy used to increase the crack, which is postulated to be proportional to the area increment of the crack itself.

For an antiplane displacement, the elastodynamics system leads to the following wave equation
\begin{equation}\label{wave_eq}
    \ddot u(t,x)-\Delta u(t,x)=f(t,x)\quad t\in[0,T],\,x\in\Omega\setminus\Gamma_t,
\end{equation}
with some prescribed boundary and initial conditions. Here, $\Omega\subset\R^d$ is an open bounded set with Lipschitz boundary, which represents the cross--section of the material, the closed set $\Gamma_t\subset\overline\Omega$ models the crack at time $t$ in the reference configuration, $u(t)\colon \Omega\setminus\Gamma_t\to\R$ is the antiplane displacement, and $f$ is a forcing term. In this case, Griffith's dynamic criterion reads
\begin{equation*}
    \mathcal E(t)+\mathcal H^{d-1}(\Gamma_t\setminus\Gamma_0)=\mathcal E(0)+\text{work of external forces},
\end{equation*}
where $\mathcal E(t)$ is the total energy at time $t$, given by the sum of kinetic and elastic energy, and $\mathcal H^{d-1}$ is the $(d-1)$--dimensional Hausdorff measure.

From the mathematical point of view, a first step to study the evolution of the fracture is to solve the wave equation~\eqref{wave_eq} when the evolution of the crack is assigned, see for example~\cite{NS,DM-Lar,DM-Luc,C2,T1} (we refer also to~\cite{DMC,DMLN,RN} for the case of a 1--dimensional model). When we want to take into account the viscoelastic properties of the material, Kelvin--Voigt's model is the most common one. If no crack is present, this leads to the damped wave equation
\begin{equation}\label{damped_eq}
    \ddot u(t,x)-\Delta u(t,x)-\Delta\dot u(t,x)=f(t,x)\quad(t,x)\in(0,T)\times\Omega.
\end{equation}
As it is well known, the solutions to~\eqref{damped_eq} satisfy the energy--dissipation balance 
\begin{equation}\label{en_diss}
    \mathcal E(t)+\int_0^t\int_\Omega|\nabla\dot u|^2\,\de x\,\de s=\mathcal E(0)+\text{work of external forces}.
\end{equation}
When we consider a crack in a viscoelastic material, Griffith's dynamic criterion becomes
\begin{equation}\label{Gr_en_diss}
    \mathcal E(t)+\mathcal H^{d-1}(\Gamma_t\setminus\Gamma_0)+\int_0^t\int_\Omega|\nabla\dot u|^2\,\de x\,\de s=\mathcal E(0)+\text{work of external forces}.
\end{equation}
For a prescribed crack evolution, this model was already considered by~\cite{DM-Lar} in the antiplane case, and more in general by~\cite{T1} for the vector--valued case. As proved in the quoted papers, the solutions to~\eqref{damped_eq} on a domain with a prescribed time--dependent crack, i.e., with $\Omega$ replaced by $\Omega\setminus\Gamma_t$, satisfy~\eqref{en_diss} for every time. This equality implies that~\eqref{Gr_en_diss} cannot be satisfied unless $\Gamma_t=\Gamma_0$ for every $t$. This phenomenon was already well known in mechanics as the viscoelastic paradox, see for instance~\cite[Chapter 7]{SL}.

To overcome this problem, we modify Kelvin--Voigt's model by considering a possibly degenerate viscosity term depending on $t$ and $x$. More precisely, we study the following equation
\begin{equation}\label{viscoelastic}
   \ddot u(t,x)-\Delta u(t,x)-\div (\Psi^2(t,x)\nabla\dot u(t,x))=f(t,x)\quad t\in[0,T],\, x\in\Omega\setminus\Gamma_t.
\end{equation}
On the function $\Psi\colon (0,T)\times\Omega\to\R$ we only require some regularity assumptions (see~\eqref{Psi}); a particularly interesting case is when $\Psi$ assumes the value zero on some points of $\Omega$, which means that the material has no longer viscoelastic properties in such a zone.

The main result of this paper is Theorem~\ref{thm:mainresult}, in which we show the existence of a weak solution to~\eqref{viscoelastic}. This is done in the more general case of linear elasticity, that is when the displacement is vector--valued and the elastic energy depends only on the symmetric part of its gradient. To this aim, we first perform a time discretization in the same spirit of~\cite{DM-Lar}, and then we pass to the limit as the time step goes to zero by relying on energy estimates; as a byproduct, we obtain the energy--dissipation inequality~\eqref{eq:enin}. By using the change of variables method implemented in~\cite{NS,DM-Luc}, we also prove a uniqueness result, but only in dimension $d=2$ and when $\Psi(t)$ vanishes on a neighborhood of the tip of $\Gamma_t$.

We complete our work by providing an example in $d=2$ of a weak solution to~\eqref{viscoelastic} for which the fracture can grow while balancing the energy. More precisely, when the cracks $\Gamma_t$ move with constant speed along the $x_1$--axis and $\Psi(t)$ is zero in a neighborhood of the crack tip, we construct a function $u$ which solves~\eqref{viscoelastic} and satisfies 
\begin{equation}\label{Psi_en_bal}
    \mathcal E(t)+\int_0^t\int_\Omega|\Psi\nabla\dot u|^2\,\de x\,\de s+\mathcal H^1(\Gamma_t\setminus\Gamma_0)=\mathcal E(0)+\text{work of external forces}.
\end{equation}
Notice that this is the natural extension of Griffith's dynamic criterion~\eqref{Gr_en_diss} to this setting.

The paper is organized as follows. In Section~\ref{visc_sec1} we fix the notation adopted throughout the paper, we list the standard assumptions on the family of cracks $\{\Gamma_t\}_{t\in[0,T]}$ and on the function $\Psi$, and we specify the notion of weak solution to problem~\eqref{viscoelastic}. In Section~\ref{visc_sec2} we state our main existence result (Theorem~\ref{thm:mainresult}), and we implement the time discretization method. We conclude the proof of Theorem~\ref{thm:mainresult} in Section~\ref{visc_sec3}, where we show the validity of the initial conditions and the energy--dissipation inequality~\eqref{eq:enin}. Section~\ref{visc_sec4} deals with uniqueness: under stronger regularity assumptions on the cracks sets, in Theorem~\ref{thm:uniq} we prove the uniqueness of a weak solution, but only when the space dimension is $d=2$. To this aim, we assume also that the function $\Psi$ is zero in a neighborhood of the crack tip. We conclude with Section~\ref{visc_sec5}, where in dimension $d=2$ we show an example of a moving crack that satisfies Griffith's dynamic energy--dissipation balance~\eqref{Psi_en_bal}.
	
\section{Notation and Preliminary Results}\label{visc_sec1}

The space of $m\times d$ matrices with real entries is denoted by $\R^{m\times d}$; in case $m=d$, the subspace of symmetric matrices is denoted by $\R^{d\times d}_{sym}$. Given two vectors $v_1,v_2\in \R^d$, their Euclidean scalar product is denoted by $v_1\cdot v_2\in\R$ and their tensor product is denoted by $v_1\otimes v_2\in\R^{d\times d}$; we use $v_1\odot v_2\in \R^{d\times d}_{sym}$ to denote the symmetric part of $v_1\otimes v_2$, namely $v_1\odot v_2:=\frac{1}{2}(v_1\otimes v_2+v_2\otimes v_1)$. 
%Given $A\in\R^{m\times d}$ and $b\in\R^d$, we write $A b\in \R^m$ to denote the vector defined as $(A b)_i:=\sum_{j=1}^d A_{ij}b_j$, $i=1,\dots m$; we always consider the elements of $\R^d$ as column vectors. 
Given $A\in\R^{m\times d}$, we use $A^T$ to denote its transpose; 
%$A_1A_2$ denotes their matrix product and 
we use $A_1\cdot A_2\in\R$ to denote the Euclidean scalar product of two matrices $A_1,A_2\in\R^{d\times d}$.
%, namely $A_1\cdot A_2:=\sum_{i,j=1}^d(A_1)_{i,j}(A_2)_{i,j}$.

The partial derivatives with respect to the variable $x_i$ are denoted by $\partial_i$. Given a function $f\colon \R^d\to\R^m$, we denote its Jacobian matrix by $\nabla f$, whose components are $(\nabla f)_{ij}:=\partial_j f_i$, $i=1,\dots,m$, $j=1,\dots,d$. For a tensor field $F\colon \R^d\to\R^{m\times d}$, by $\div F$ we mean the divergence of $F$ with respect to rows, namely $(\div F)_i:=\sum_{j=1}^d\partial_jF_{ij}$, for $i=1,\dots,m$. 

The $d$--dimensional Lebesgue measure is denoted by $\mathcal L^d$ and the $(d-1)$--dimensional Hausdorff measure by $\mathcal H^{d-1}$. We adopted standard notations for Lebesgue and Sobolev spaces on open subsets of $\R^d$; given an open set $\Omega\subseteq\mathbb R^d$ we use $\norm{\cdot}_\infty$ to denote the norm of $L^\infty(\Omega;\R^m)$. The boundary values of a Sobolev function are always intended in the sense of traces. Given a bounded open set $\Omega$ with Lipschitz boundary, we denote by $\nu$ the outer unit normal vector to $\partial\Omega$, which is defined $\mathcal H^{d-1}$--a.e. on the boundary.

Given a Banach space $X$, its norm is denoted by $\Vert\cdot\Vert_X$; if $X$ is an Hilbert space, we use $(\cdot,\cdot)_X$ to denote its scalar product. The dual space of $X$ is denoted by $X'$, and we use $\spr{\cdot}{\cdot}_{X'}$ to denote the duality product between $X'$ and $X$.  Given two Banach spaces $X_1$ and $X_2$, the space of linear and continuous maps from $X_1$ to $X_2$ is denoted by $\mathscr L(X_1;X_2)$; given $\mathbb A\in\mathscr L(X_1;X_2)$ and $u\in X_1$, we write $\mathbb A u\in X_2$ to denote the image of $u$ under $\mathbb A$. 

Given an open interval $(a,b)\subseteq\R$, $L^p(a,b;X)$ is the space of $L^p$ functions from $(a,b)$ to $X$. Given $u\in L^p(a,b;X)$, we denote by $\dot u\in\mathcal D'(a,b;X)$ its distributional derivative. The set of continuous functions from $[a,b]$ to $X$ is denoted by $C^0([a,b];X)$. Given a reflexive Banach space $X$, $C_w^0([a,b];X)$ is the set of weakly continuous functions from $[a,b]$ to $X$, namely
$$C_w^0([a,b];X):=\{u\colon [a,b]\to X: t\mapsto \spr{x'}{u(t)}_{X'}\text{ is continuous from $[a,b]$ to $\R$ for every }x'\in X'\}.$$

Let $T$ be a positive real number and let $\Omega\subset\R^d$ be a bounded open set with Lipschitz boundary. Let $\partial_D\Omega$ be a (possibly empty) Borel subset of $\partial\Omega$ and let $\partial_N\Omega$ be its complement. We assume the following hypotheses on the geometry of the cracks:
\begin{itemize}
\item[(E1)] $\Gamma\subset\overline\Omega$ is a closed set with $\mathcal L^d(\Gamma)=0$ and $\mathcal H^{d-1}(\Gamma\cap\partial\Omega)=0$;
\item[(E2)] for every $x\in\Gamma$ there exists an open neighborhood $U$ of $x$ in $\R^d$ such that $(U\cap\Omega)\setminus\Gamma$ is the union of two disjoint open sets $U^+$ and $U^-$ with Lipschitz boundary;
\item[(E3)] $\{\Gamma_t\}_{t\in[0,T]}$ is a family of closed subsets of $\Gamma$ satisfying $\Gamma_s\subset\Gamma_t$ for every $0\le s\le t\le T$.
\end{itemize}
Thanks (E1)--(E3) the space $L^2(\Omega\setminus\Gamma_t;\R^m)$ coincides with $L^2(\Omega;\R^m)$ for every $t\in[0,T]$ and $m\in\mathbb N$. In particular, we can extend a function $u\in L^2(\Omega\setminus\Gamma_t;\R^m)$ to a function in $L^2(\Omega;\R^m)$ by setting $u=0$ on $\Gamma_t$. Moreover, the trace of $u\in H^1(\Omega\setminus\Gamma)$ is well defined on $\partial\Omega$. Indeed, we may find a finite number of open sets with Lipschitz boundary $U_j\subset\Omega\setminus\Gamma$, $j=1,\dots m$, such that $\partial\Omega\setminus(\Gamma\cap\partial\Omega)\subset\cup_{j=1}^m\partial U_j$. Since $\mathcal H^{d-1}(\Gamma\cap\partial\Omega)=0$, there exists a constant $C>0$, depending only on $\Omega$ and $\Gamma$, such that
\begin{equation}\label{eq:vis_trace}
\norm{u}_{L^2(\partial\Omega)}\le C\norm{u}_{H^1(\Omega\setminus\Gamma)}\quad\text{for every }u\in H^1(\Omega\setminus\Gamma;\R^d).
\end{equation}
Similarly, we can find a finite number of open sets $U_j\subset\Omega\setminus\Gamma$, $j=1,\dots m$, with Lipschitz boundary, such that $\Omega\setminus\Gamma=\cup_{j=1}^m U_j$. By using second Korn's inequality in each $U_j$ (see, e.g.,~\cite[Theorem 2.4]{OSY}) and taking the sum over $j$ we can find a constant $C_K$, depending only on $\Omega$ and $\Gamma$, such that 
\begin{equation}\label{eq:hyp_korn}
\norm{\nabla u}_{L^2(\Omega;\R^{d\times d})}^2\le C_K\left(\norm{u}_{L^2(\Omega;\R^d)}^2+\norm{\E u}_{L^2(\Omega;\R^{d\times d}_{sym})}^2\right)\quad\text{for every }u\in H^1(\Omega\setminus\Gamma;\R^d),
\end{equation}
where $\E u$ is the symmetric part of $\nabla u$, i.e., $\E u\coloneqq \frac{1}{2}( \nabla u+ \nabla u^T)$.

For every $t\in[0,T]$ we define
\begin{equation*}
V_t:=\{u\in L^2(\Omega\setminus\Gamma_t;\R^d):\E u\in L^2(\Omega\setminus\Gamma_t;\R^{d\times d}_{sym})\}.
\end{equation*}
Notice that in the definition of $V_t$ we are considering only the distributional gradient of $u$ in $\Omega\setminus\Gamma_t$ and not the one in $\Omega$. The set $V_t$ is a Hilbert space with respect to the following norm
\begin{equation*}
\norm{u}_{V_t}:=(\norm{u}_{H}^2+\norm{\E u}_{H}^2)^{\frac{1}{2}}\quad\text{for every }u\in V_t.
\end{equation*}
To simplify our exposition, for every $m\in\mathbb N$ we set $H:=L^2(\Omega;\R^m)$ and $H_N:=L^2(\partial_N\Omega;\R^m)$; we always identify the dual of $H$ by $H$ itself and $L^2(0,T;L^2(\Omega;\R^m))$ by $L^2((0,T)\times\Omega;\R^m)$.

Thanks to~\eqref{eq:hyp_korn}, the space $V_t$ coincides with the usual Sobolev space $H^1(\Omega\setminus\Gamma_t;\R^d)$. 
Therefore, by~\eqref{eq:vis_trace}, it makes sense to consider for every $t\in[0,T]$ the set
\begin{equation*}
V_t^D:=\{u\in V_t\,:u=0\text{ on }\partial_D\Omega\},
\end{equation*}
which is a Hilbert space with respect to $\norm{\cdot}_{V_t}$. Moreover, by combining~\eqref{eq:hyp_korn} with~\eqref{eq:vis_trace}, we derive also the existence of a constant $C_{tr}>0$ such that 
\begin{equation}\label{trace}
\norm{u}_{H_N}\leq C_{tr}\norm{u}_{V_T}\quad\text{for every }u\in V_T.
\end{equation}

Let $\mathbb C,\mathbb B\colon \Omega\to \mathscr L(\R^{d\times d}_{sym};\R^{d\times d}_{sym})$ be two fourth-order tensors satisfying:
\begin{align}
&\mathbb C_{ijhk},\mathbb B_{ijhk}\in L^\infty(\Omega)\quad\text{for every $i,j,h,k=1,\dots,d$},\label{CB1}\\
& \mathbb C(x)\eta_1\cdot\eta_2=\eta_1\cdot \mathbb C(x)\eta_2,\quad\mathbb B(x)\eta_1\cdot\eta_2=\eta_1\cdot \mathbb B(x)\eta_2\quad\text{for a.e. $x\in\Omega$ and for every $\eta_1,\eta_2\in\R_{sym}^{d\times d}$},\\
&\mathbb C(x)\eta\cdot\eta\ge \lambda_1|\eta|^2,\quad\mathbb B(x)\eta\cdot\eta\ge \lambda_2|\eta|^2\quad\text{for a.e. $x\in\Omega$ and for every $\eta\in\R_{sym}^{d\times d}$}\label{CB2},
\end{align}
for two positive constants $\lambda_1,\lambda_2$ independent of $x$. Consider a function $\Psi\colon(0,T)\times\Omega\to\R$ satisfying 
\begin{equation}\label{Psi}
\Psi\in L^\infty ((0,T)\times\Omega),\quad \nabla\Psi\in L^\infty((0,T)\times\Omega;\R^d).
\end{equation} 
Given $f\in L^2(0,T;H)$, $w\in H^2(0,T;H)\cap H^1(0,T;V_0)$, $g\in H^1(0,T;H_N)$, $u^0\in V_0$ with $u^0-w(0)\in V_0^D$, and $u^1\in H$, we want to find a solution to the viscoelastic dynamic system
\begin{equation}\label{visco_elasto_system}
    \ddot u(t)-\div(\mathbb C\E u(t))-\div(\Psi^2(t)\mathbb B \E \dot u(t))=f(t)\quad\text{in $\Omega\setminus\Gamma_t$, $t\in(0,T)$},
\end{equation}
satisfying the following boundary and initial conditions
\begin{align}
&u(t)=w(t) & &\hspace{-2.4cm}\text{on $\partial_D\Omega$, $t\in(0,T)$},\label{boundary0} & &\\
&(\mathbb C\E u(t)+\Psi^2(t)\mathbb B \E \dot u(t))\nu=g(t) & &\hspace{-2.4cm}\text{on $\partial_N\Omega$,  $t\in(0,T)$}, & &\label{boundary1}\\
& (\mathbb C\E u(t)+\Psi^2(t)\mathbb B \E \dot u(t))\nu=0 & &\hspace{-2.4cm}\text{on $\Gamma_t$, \quad $t\in(0,T)$},\label{boundary} & & \\
&   u(0)=u^0,\quad \dot u(0)=u^1.\label{initials} 
\end{align}
As usual, the Neumann boundary conditions are only formal, and their meaning will be specified in Definition~\ref{def:weaksol}.

Throughout the paper we always assume that the family $\{\Gamma_t\}_{t\in[0,T]}$  satisfies (E1)--(E3), as well as $\mathbb C$, $\mathbb B$, $\Psi$, $f$, $w$, $g$, $u^0$, and $u^1$ the previous hypotheses. Let us define the following functional spaces:
\begin{align*}
&\mathcal V:=\{\varphi\in L^2(0,T;V_T):\dot\varphi\in L^2(0,T;H),\,\varphi(t)\in V_t\text{ for a.e. $t\in(0,T)$}\},\\
&\mathcal V^D:=\{\varphi\in\mathcal V:\varphi(t)\in V_t^D\text{ for a.e. $t\in(0,T)$}\},\\
&\mathcal W:=\{u\in\mathcal V:\Psi \dot u\in L^2(0,T;V_T),\,\Psi(t)\dot u(t)\in V_t\text{ for a.e. $t\in(0,T)$} \}.
\end{align*}

\begin{remark}
In the classical viscoelastic case, namely when $\Psi$ is identically equal to $1$, the solution $u$ to system~\eqref{visco_elasto_system} has derivative $\dot u(t)\in V_t$ for a.e. $t\in(0,T)$ with $\E \dot u\in L^2(0,T;H)$. For a generic $\Psi$ we expect to have $\Psi \E \dot u\in L^2(0,T;H)$. Therefore $\mathcal W$ is the natural setting where looking for a solution to~\eqref{visco_elasto_system}. Indeed, from a distributional point of view we have 
$$\Psi (t)\E \dot u(t)=\E (\Psi(t) \dot u(t))-\nabla\Psi(t)\odot \dot u(t)\quad\text{in $\mathcal D'(\Omega\setminus\Gamma_t;\R^{d\times d}_{sym})$ for a.e. $t\in(0,T)$},$$
and $\E (\Psi\dot u),\nabla\Psi\odot \dot u\in L^2(0,T;H)$ if $u\in\mathcal W$, thanks to~\eqref{Psi}.
\end{remark}

\begin{remark}
The set $\mathcal W$ coincides with the space of functions $u\in H^1(0,T;H)$ such that $u(t)\in V_t$ and $\Psi(t)\dot u(t)\in V_t$ for a.e. $t\in(0,T)$, and satisfying 
\begin{equation}\label{secverW}
\int_0^T \norm{u(t)}^2_{V_t}+\norm{\Psi(t)\dot u(t)}_{V_t}^2 \,\de t<\infty.
\end{equation}
This is a consequence of the strong measurability of the maps $t\mapsto u(t)$ and $t\mapsto \Psi(t)\dot u(t)$ from $(0,T)$ into $V_T$, which gives that~\eqref{secverW} is well defined and $u,\Psi\dot u\in L^2(0,T;V_T)$. To prove the strong measurability of these two maps, it is enough to observe that $V_T$ is a separable Hilbert space and that the maps $t\mapsto \dot u(t)$ and $t\mapsto \Psi(t)\dot u(t)$ from $(0,T)$ into $V_T$ are weakly measurable. Indeed, for every $\varphi\in C^\infty_c(\Omega\setminus\Gamma_T)$ the maps
\begin{align*}
&t\mapsto \int_{\Omega\setminus\Gamma_T}\E  u(t,x) \varphi(x)\,\de x=-\int_{\Omega\setminus\Gamma_T} u(t,x)\odot\nabla \varphi(x)\,\de x,\\
&t\mapsto \int_{\Omega\setminus\Gamma_T}\E (\Psi(t,x)\dot u(t,x))\varphi(x)\,\de x=-\int_{\Omega\setminus\Gamma_T}\Psi(t,x)\dot u(t,x)\odot\nabla \varphi(x)\,\de x
\end{align*}
are measurable from $(0,T)$ into $\R$, and $C^\infty_c(\Omega\setminus\Gamma_T)$ is dense in $L^2(\Omega)$.
\end{remark}

\begin{lemma}
The spaces $\mathcal V$ and $\mathcal W$ are Hilbert spaces with respect to the following norms:
\begin{align*}
    &\norm{\varphi}_{\mathcal V}:=(\norm{\varphi}^2_{L^2(0,T;V_T)}+\norm{\dot \varphi}^2_{L^2(0,T;H)})^{\frac{1}{2}}\quad\text{for every }\varphi\in\mathcal V,\\
    &\norm{u}^2_{\mathcal W}:=(\norm{u}_{\mathcal V}+\norm{\Psi\dot u}^2_{L^2(0,T;V_T)})^{\frac{1}{2}}\quad\text{for every }u\in\mathcal W.
\end{align*}
Moreover, $\mathcal V^D$ is a closed subspace of $\mathcal V$.
\end{lemma}

\begin{proof}
It is clear that $\norm{\cdot}_{\mathcal V}$ and $\norm{\cdot}_{\mathcal W}$ are norms on $\mathcal V$ and $\mathcal W$ induced by scalar products. We just have to check the completeness of such spaces with respect to these norms. 

Let $\{\varphi_k\}_k\subset \mathcal V$ be a Cauchy sequence. Then, $\{\varphi_k\}_k$ and $\{\dot \varphi_k\}_k$ are Cauchy sequences, respectively, in $L^2(0,T;V_T)$ and $L^2(0,T;H)$, which are complete Hilbert spaces. Thus there exists $\varphi\in L^2(0,T;V_T)$ with $\dot \varphi\in L^2(0,T;H)$ such that $\varphi_k\to \varphi$ in $L^2(0,T;V_T)$ and $\dot \varphi_k\to\dot \varphi$ in $L^2(0,T;H)$. In particular there exists a subsequence $\{\varphi_{k_j}\}_j$ such that $\varphi_{k_j}(t)\to \varphi(t)$ in $V_T$ for a.e. $t\in(0,T)$. Since $\varphi_{k_j}(t)\in V_t$ for a.e. $t\in(0,T)$ we deduce that $\varphi(t)\in V_t$ for a.e. $t\in(0,T)$. Hence $\varphi\in\mathcal V$ and $\varphi_k\to \varphi$ in $\mathcal V$. With a similar argument, we can prove that $\mathcal V^D\subset \mathcal V$ is a closed subspace.

Let us now consider a Cauchy sequence $\{u_k\}_k\subset \mathcal W$. We have that $\{u_k\}_k$ and $\{\Psi\dot u_k\}_k$ are Cauchy sequences, respectively, in $\mathcal V$ and $L^2(0,T;V_T)$,  which are complete Hilbert spaces. Thus there exist two functions $u\in \mathcal V$ and $z\in L^2(0,T;V_T)$ such that $u_k\to u$ in $\mathcal V$ and $\Psi\dot u_k\to z$ in $L^2(0,T;V_T)$. Since $\dot u_k\to \dot u$ in $L^2(0,T;H)$ and $\Psi\in L^\infty((0,T)\times\Omega)$, we also have that $\Psi\dot u_k\to\Psi\dot u$ in $L^2(0,T;H)$, which gives that $z=\Psi\dot u$. Finally let us prove that $\Psi(t)\dot u(t)\in V_t$ for a.e. $t\in(0,T)$. By the fact that $\Psi\dot u_k\to\Psi\dot u$ in $L^2(0,T;V_T)$, there exists a subsequence $\{\Psi\dot u_{k_j}\}_j$ such that $\Psi(t)\dot u_{k_j}(t)\to \Psi(t)\dot u(t)$ in $V_T$ for a.e. $t\in(0,T)$. Since $\Psi(t)\dot u_{k_j}(t)\in V_t$ for a.e. $t\in(0,T)$ we deduce that $\Psi(t)\dot u(t)\in V_t$ for a.e. $t\in(0,T)$. Hence $u\in\mathcal W$ and $u_k\to u$ in $\mathcal W$.
\end{proof}

We are now in position to define a weak solution to~\eqref{visco_elasto_system}--\eqref{boundary}.

\begin{definition}[Weak solution]\label{def:weaksol}
We say that $u\in\mathcal W$ is a {\it weak solution} to system~\eqref{visco_elasto_system} with boundary conditions~\eqref{boundary0}--\eqref{boundary} if $u-w\in\mathcal V^D$ and 
\begin{equation}\label{weakform}
\begin{aligned}
    &-\int_0^T(\dot u(t),\dot\varphi(t))_H \,\de t+\int_0^T(\mathbb C\E u(t),\E \varphi(t))_{H} \,\de t+\int_0^T(\mathbb B\E (\Psi(t)\dot u(t)),\Psi(t)\E \varphi(t))_{H} \,\de t\\
    &-\int_0^T(\mathbb B\nabla\Psi(t)\odot \dot u(t),\Psi(t)\E \varphi(t))_{H} \,\de t =\int_0^T(f(t),\varphi(t))_H \,\de t+\int_0^T(g(t),\varphi(t))_{H_N} \,\de t
    \end{aligned}
\end{equation}
for every $\varphi\in\mathcal V^D$ such that $\varphi(0)=\varphi(T)=0$.
\end{definition}
Notice that the Neumann boundary conditions~\eqref{boundary1} and~\eqref{boundary} can be obtained from~\eqref{weakform}, by using integration by parts in space, only when $u(t)$ and $\Gamma_t$ are sufficiently regular.

\begin{remark}
If $\dot u$ is regular enough (for example $\dot u\in L^2(0,T;V_T)$ with $\dot u(t)\in V_t$ for a.e. $t\in(0,T)$), then we have $\Psi \E \dot u=\E (\Psi \dot u)-\nabla\Psi\odot \dot u$. Therefore~\eqref{weakform} is coherent with the strong formulation~\eqref{visco_elasto_system}. In particular, for a function $u\in\mathcal W$ we can define
\begin{equation}\label{Psiedotu}
        \Psi \E \dot u:=\E (\Psi \dot u)-\nabla\Psi\odot \dot u\in L^2(0,T;H),
\end{equation}
so that equation~\eqref{weakform} can be rephrased as 
\begin{equation*}
\begin{aligned}
    &-\int_0^T(\dot u(t),\dot\varphi(t))_H \,\de t+\int_0^T(\mathbb C\E u(t),\E \varphi(t))_{H} \,\de t
    +\int_0^T(\mathbb B\Psi(t)\E \dot u(t),\Psi(t)\E \varphi(t))_{H} \,\de t\\
    &=\int_0^T(f(t),\varphi(t))_H \,\de t+\int_0^T(g(t),\varphi(t))_{H_N} \,\de t
    \end{aligned}
\end{equation*}
for every $\varphi\in\mathcal V^D$ such that $\varphi(0)=\varphi(T)=0$.
\end{remark}

\begin{definition}[Initial conditions]
We say that $u\in\mathcal W$ satisfies the initial conditions~\eqref{initials} if
\begin{equation}\label{icd}
    \lim_{h\to 0^+}\frac{1}{h}\int_0^h(\norm{u(t)-u^0}_{V_t}^2+\norm{\dot u(t)-u^1}_H^2) \,\de t=0.
\end{equation}
\end{definition}

\section{Existence}\label{visc_sec2}
We now state our main existence result, whose proof will be given at the end of Section~\ref{visc_sec3}.

\begin{theorem}\label{thm:mainresult}
	There exists a weak solution $u\in\mathcal W$ to~\eqref{visco_elasto_system}--\eqref{boundary} satisfying the initial conditions $u(0)=u^0$ and $\dot u(0)=u^1$ in the sense of~\eqref{icd}. Moreover $u\in C_w([0,T];V_T)$, $\dot u\in C_w([0,T];H)\cap H^1(0,T;(V_0^D)')$, and
	\begin{equation*}
	\lim_{t\to 0^+}u(t)= u^0\text{ in $V_T$},\quad\lim_{t\to 0^+}\dot u(t)=u^1\text{ in $H$}.
	\end{equation*}
\end{theorem}

To prove the existence of a weak solution to~\eqref{visco_elasto_system}--\eqref{boundary}, we use a time discretization scheme in the same spirit of~\cite{DM-Lar}.
Let us fix $n\in\mathbb N$ and set
\begin{equation*}
    \tau_n:=\frac{T}{n},\quad u_n^0:=u^0,\quad u_n^{-1}:=u^0-\tau_nu^1. 
\end{equation*}
We define
\begin{align*}
    &V_n^k:=V_{k\tau_n}^D,\quad g_n^k:=g(k\tau_n),\quad w_n^k:=w(k\tau_n)\quad\text{for }k=0,\dots,n,\\
    &f_n^k:=\frac{1}{\tau_n}\int_{(k-1)\tau_n}^{k\tau_n} f(s)\,\de s,\quad \Psi_n^k:=\frac{1}{\tau_n}\int_{(k-1)\tau_n}^{k\tau_n}\Psi(s)\,\de s,\quad\delta g_n^k:=\frac{g_n^k-g_n^{k-1}}{\tau_n}\quad \text{for }k=1,\dots,n,\\
    &\delta w_n^0:=\dot w(0),\quad\delta w_n^k:=\frac{w_n^k- w_n^{k-1}}{\tau_n},\quad \delta^2 w_n^k:=\frac{\delta w_n^k-\delta w_n^{k-1}}{\tau_n}\quad \text{for }k=1,\dots,n.
\end{align*}
For every $k=1,\dots,n$ let $u_n^k\in V_T$, with $u_n^k-w_n^k\in V_n^k$, be the solution to 
\begin{equation}\label{unk}
(\delta^2u_n^k,v)_H+(\mathbb C\E u_n^k,\E v)_{H}+(\mathbb B\Psi_n^k \E \delta u_n^k,\Psi_n^k \E v)_{H}=(f_n^k,v)_H+(g_n^k,v)_{H_N}\quad\text{for every }v\in V_n^k,
\end{equation}
where 
\begin{align*}
    &\delta u_n^k:=\frac{u_n^k- u_n^{k-1}}{\tau_n}\quad \text{for }k=0,\dots,n,\quad \delta^2u_n^k:=\frac{\delta u_n^k-\delta u_n^{k-1}}{\tau_n}\quad \text{for }k=1,\dots,n.
\end{align*}
The existence of a unique solution $u_n^k$ to~\eqref{unk} is an easy application of Lax--Milgram's theorem.

\begin{remark}
Since $\delta u_n^k\in V_{(k-1)\tau_n}$, then $\Psi_n^k\E \delta u_n^k=\E (\Psi_n^ku_n^k)-\nabla\Psi_n^k\odot u_n^k$, so that the discrete equation~\eqref{unk} is coherent with the weak formulation given in~\eqref{weakform}.
\end{remark}

In the next lemma, we show a uniform estimate for the family $\{u_n^k\}_{k=1}^n$ with respect to $n\in\mathbb N$ that will be used later to pass to the limit in the discrete equation~\eqref{unk}.

\begin{lemma}\label{lem:est}
There exists a constant $C>0$, independent of $n\in\mathbb N$, such that 
\begin{equation}\label{eq:est}
\max_{i=1,..,n}\|\delta u_n^i\|_H+\max_{i=1,..,n}\norm{\E u_n^i}_H+\sum_{i=1}^n \tau_n \norm{\Psi_n^i \E \delta u_n^i}_H^2\le C.
\end{equation}
\end{lemma}

\begin{proof}
We fix $n\in\mathbb N$. To simplify the notation we set 
$$a(u,v):=(\mathbb C\E u,\E v)_{H},\quad b_n^k(u,v):=(\mathbb B\Psi_n^k \E u,\Psi_n^k \E v)_{H}\quad\text{for every $u,v\in V_T$}.$$ 
By taking as test function $v=\tau_n(\delta u_n^k-\delta w_n^k)\in V_n^k$ in~\eqref{unk}, for $k=1,\dots,n$ we obtain
% \begin{equation}
% \begin{aligned}
% &\tau_n(\delta^2 u_n^k,\delta u_n^k-\delta w_n^k)_H+\tau_n a(u_n^k,\delta u_n^k-\delta w_n^k)+\tau_n b_n^k(\delta u_n^k,\delta u_n^k-\delta w_n^k)\\
% &=\tau_n(f_n^k,\delta u_n^k-\delta w_n^k)_H+\tau_n(g_n^k,\delta u_n^k-\delta w_n^k)_{H_N},
% \end{aligned}
% \end{equation}
% which can be rewritten as
\begin{align*}
&\|\delta u_n^k\|^2_H-(\delta u_n^{k-1},\delta u_n^{k})_H+a(u_n^k,u_n^k)-a(u_n^k,u_n^{k-1})+\tau_n b_n^k(\delta u_n^k,\delta u_n^k)=\tau_n L_n^k,
\end{align*}
where
\begin{equation*}
L_n^k:=(f_n^k,\delta u_n^k-\delta w_n^k)_H+(g_n^k,\delta u_n^k-\delta w_n^k)_{H_N}+(\delta^2 u_n^{k},\delta w_n^{k})_H+a(u_n^k,\delta w_n^k)+b_n^k(\delta u_n^k,\delta w_n^k).
\end{equation*}
Thanks to the following identities 
\begin{align*}
&\norm{\delta u_n^k}_H^2-(\delta u_n^{k-1},\delta u_n^k)_H=\frac{1}{2}\norm{\delta u_n^k}_H^2-\frac{1}{2}\norm{\delta u_n^{k-1}}_H^2+\frac{\tau_n^2}{2}\norm{\delta^2 u_n^k}_H^2,\\
&a(u_n^k,u_n^k)-a(u_n^k,u_n^{k-1})=\frac{1}{2}a(u_n^k,u_n^k)-\frac{1}{2}a(u_n^{k-1},u_n^{k-1})+\frac{\tau_n^2}{2}a(\delta u_n^k,\delta u_n^k),
\end{align*}
and by omitting the terms with $\tau_n^2$, which are non negative, we derive
% \begin{align*}
% &\frac{\|\delta u_n^k\|^2_H}{2}-\frac{\|\delta u_n^{k-1}\|^2_H}{2}+\frac{a(u_n^k,u_n^k)}{2}-\frac{a(u_n^{k-1},u_n^{k-1})}{2}+\tau_n b_n^k(\delta u_n^{k},\delta u_n^{k})\\
% &+\frac{\tau_n^2}{2}\left(\|\delta^2 u_n^k\|^2_H+a(\delta u_n^{k},\delta u_n^{k})\right)=\tau_n L^n_k,
% \end{align*}
% and in particular that
\begin{equation*}
\frac{1}{2}\|\delta u_n^k\|^2_H-\frac{1}{2}\|\delta u_n^{k-1}\|^2_H+\frac{1}{2}a(u_n^k,u_n^k)-\frac{1}{2}a(u_n^{k-1},u_n^{k-1})+\tau_n b_n^k(\delta u_n^{k},\delta u_n^{k})\leq\tau_n L_n^k.
\end{equation*}
We fix $i\in\{1,\dots,n\}$ and sum over $k=1,\dots, i$ to obtain the following discrete energy inequality
\begin{equation}\label{grw}
\frac{1}{2}\|\delta u_n^i\|^2_H+\frac{1}{2}a(u_n^i,u_n^i)+\sum_{k=1}^i\tau_n b_n^k(\delta u_n^{k},\delta u_n^{k})\leq \mathcal{E}_0+\sum_{k=1}^i\tau_n L_n^k,
\end{equation}
where $\mathcal E_0:=\frac{1}{2}\norm{u^1}_H^2+\frac{1}{2}(\mathbb C\E u^0,\E u^0)_{H}$. Let us now estimate the right--hand side in~\eqref{grw} from above. By~\eqref{trace} and~\eqref{CB1} we have
\begin{align}
&\left|\sum_{k=1}^i \tau_n(f_n^k,\delta u_n^k-\delta w_n^k)_H\right| 
% & \leq  \sum_{k=1}^i \tau_n\|f_n^k\|_H\|\delta u_n^k\|_H\leq \frac{1}{2}\sum_{k=1}^i \tau_n\|f_n^k\|^2_H+\frac{1}{2}\sum_{k=1}^i \tau_n\|\delta u_n^k\|^2_H
% \nonumber\\
% &\leq \frac{1}{2}\sum_{k=1}^n \tau_n\|f_n^k\|^2_H+\frac{1}{2}\sum_{k=1}^i \tau_n\|\delta u_n^k\|^2_H
\leq \|f\|^2_{L^2(0,T;H)}+\frac{1}{2}\|\dot{w}\|^2_{L^2(0,T;H)}+\frac{1}{2}\sum_{k=1}^i \tau_n\|\delta u_n^k\|^2_H,\\
% &\left|\sum_{k=1}^i \tau_n(f_n^k,\delta w_n^k)_H\right|
% &\leq \sum_{k=1}^i \tau_n\|f_n^k\|_H\|\delta w_n^k\|_H\leq \frac{1}{2}\sum_{k=1}^i \tau_n\|f_n^k\|^2_H+\frac{1}{2}\sum_{k=1}^i \tau_n\|\delta w_n^k\|^2_H\nonumber\\
% &\leq \frac{1}{2}\sum_{k=1}^n \tau_n\|f_n^k\|^2_H+\frac{1}{2}\sum_{k=1}^n \tau_n\|\delta w_n^k\|^2_H
% \leq \frac{1}{2}\|f\|^2_{L^2(0,T;H)}
% \end{align}
% Moreover by~\eqref{trace} and (E6)
% \begin{align}
&\left|\sum_{k=1}^i \tau_n a(u_n^k,\delta w_n^k)\right|
% &\leq \sum_{k=1}^i \tau_n\|\mathbb{C} \|_\infty\|\E u_n^k\|_H\|\E \delta w_n^k\|_H \leq \frac{\|\mathbb{C} \|_\infty}{2}\sum_{k=1}^i \tau_n\|\E u_n^k\|^2_H+\frac{\|\mathbb{C} \|_\infty}{2}\sum_{k=1}^n \tau_n\|\delta w_n^k\|^2_H\nonumber \\
\leq \frac{\|\mathbb{C} \|_\infty}{2}\| \dot{ w}\|^2_{L^2(0,T;V_0)}+\frac{\|\mathbb{C} \|_\infty}{2}\sum_{k=1}^i \tau_n\|\E u_n^k\|^2_H,\\
&\left|\sum_{k=1}^i \tau_n(g_n^k,\delta w_n^k)_{H_N}\right|
% &\leq \sum_{k=1}^i \tau_n\|g_n^k\|_{H_N}\|\delta w_n^k\|_{H_N}\leq \frac{1}{2}\sum_{k=1}^i \tau_n\|g_n^k\|^2_{H_N}+\frac{1}{2}\sum_{k=1}^i \tau_n\|\delta w_n^k\|^2_{H_N}\nonumber\\
% \leq \frac{1}{2}\sum_{k=1}^n \tau_n\|g_n^k\|^2_{H_N}+\frac{1}{2}\sum_{k=1}^n \tau_n\|\delta w_n^k\|^2_{H_N}
% \leq \frac{1}{2}\|g\|^2_{L^2(0,T;H_N)}+\frac{C_{tr}^2}{2}\sum_{k=1}^n \tau_n\|\delta w_n^k\|^2_{V_0}\nonumber\\ 
%     &
\leq \frac{1}{2}\|g\|^2_{L^2(0,T;H_N)}+\frac{C_{tr}^2}{2}\|\dot{w}\|^2_{L^2(0,T;V_0)}.
\end{align}
For the other term involving $g_n^k$, we perform the following discrete integration by parts 
\begin{align}
\sum_{k=1}^i \tau_n (g_n^k, \delta u_n^k)_{H_N}
% =\sum_{k=1}^i (g_n^k, u_n^k)_{H_N}-\sum_{k=1}^i  (g_n^k, u_n^{k-1})_{H_N}=\sum_{k=2}^{i+1} (g_n^{k-1}, u_n^{k-1})_{H_N}-\sum_{k=1}^i  (g_n^k, u_n^{k-1})_{H_N}\nonumber\\
% &=\sum_{k=1}^{i} (g_n^{k-1}, u_n^{k-1})_{H_N}-(g_n^0,u_n^0)_{H_N}+(g_n^i,u_n^i)_{H_N}-\sum_{k=1}^i  (g_n^k, u_n^{k-1})_{H_N}\nonumber\\
&=(g_n^i,u_n^i)_{H_N}-(g(0),u^0)_{H_N}-\sum_{k=1}^{i}\tau_n (\delta g_n^{k}, u_n^{k-1})_{H_N}\label{gn}.
\end{align}
Hence for every $\epsilon\in(0,1)$, by using~\eqref{trace} and Young's inequality, we get
\begin{equation}\label{mod}
\begin{aligned}
\left|\sum_{k=1}^i \tau_n (g_n^k, \delta u_n^k)_{H_N}\right|
% \leq\|g_n^i\|_{H_N}\|u_n^i\|_{H_N}+\|g(0)\|_{H_N}\|u^0\|_{H_N}+\sum_{k=1}^i \tau_n \|\delta g_n^k\|_{H_N}\|u_n^{k-1}\|_{H_N} \nonumber\\
&\leq\frac{\epsilon}{2}\|u_n^i\|^2_{H_N}+\frac{1}{2\epsilon}\|g\|^2_{L^{\infty}(0,T;H_N)}+\norm{g(0)}_{H_N}\norm{u^0}_{H_N}+\sum_{k=1}^i \tau_n \|\delta g_n^k\|_{H_N} \|u_n^{k-1}\|^2_{H_N}\\
% &\leq\frac{\epsilon}{2}\|u_n^i\|^2_{H_N}+\frac{1}{2\epsilon}\|g\|^2_{L^{\infty}(0,T;H_N)}+\|g(0)\|_{H_N}\|u^0\|_{H_N}+\frac{1}{2}\sum_{k=1}^n \tau_n \|\delta g_n^k\|^2_{H_N}+\frac{1}{2}\sum_{k=0}^{i-1} \tau_n \|u_n^{k}\|^2_{H_N}\nonumber\\
&\leq  C_\e+\frac{\epsilon C_{tr}^2}{2}\|u_n^i\|^2_{V_T}+\frac{C_{tr}^2}{2}\sum_{k=1}^{i} \tau_n \|u_n^{k}\|^2_{V_T},
\end{aligned}
\end{equation}
where $C_\epsilon$ is a positive constant depending on $\epsilon$. Thanks to Jensen's inequality we can write
\begin{equation*}
\|u_n^l\|^2_{V_T}\leq\|\E u_n^l\|^2_H+\left(\norm{u_0}_H+\sum_{j=1}^l\tau_n\norm{\delta u_n^j}_H\right)^2\leq\|\E u_n^l\|^2_H+ 2\norm{u^0}_H^2+2T\sum_{j=1}^l\tau_n\norm{\delta u_n^j}_H^2,
\end{equation*} 
so that~\eqref{mod} can be further estimated as
\begin{equation}
\begin{aligned}
\left|\sum_{k=1}^i \tau_n (g_n^k, \delta u_n^k)_{H_N}\right|&\leq C_\epsilon+\frac{\epsilon C_{tr}^2}{2}\left(\|\E u_n^i\|^2_H+ 2\norm{u^0}_H^2+2T\sum_{j=1}^i\tau_n\norm{\delta u_n^j}_H^2\right)\\
&\quad+\frac{C_{tr}^2}{2}\sum_{k=1}^{i} \tau_n \left(\|\E u_n^k\|^2_H+ 2\norm{u^0}_H^2+2T\sum_{j=1}^k\tau_n\norm{\delta u_n^j}_H^2\right)\\
&\le  \tilde C_\e+\frac{\epsilon C_{tr}^2}{2}\|\E u_n^i\|^2_H+\tilde C\sum_{k=1}^{i} \tau_n \left(\norm{\delta u_n^k}_H^2+\|\E u_n^k\|^2_H\right),
\end{aligned}
\end{equation}
for some positive constants $\tilde{C}_\epsilon$ and $\tilde{C}$, with $\tilde{C}_\epsilon$ depending on $\epsilon$. Similarly to~\eqref{gn}, we can say
\begin{equation}\label{dwn}
\sum_{k=1}^i \tau_n (\delta^2 u_n^k, \delta w_n^k)_{H}
% =&\sum_{k=1}^i (\delta u_n^k, \delta w_n^k)_{H}-\sum_{k=1}^i  (\delta u_n^{k-1}, \delta w_n^{k})_{H}=\sum_{k=2}^{i+1} (\delta u_n^{k-1}, \delta w_n^{k-1})_{H}-\sum_{k=1}^i  (\delta u_n^{k-1}, \delta w_n^{k})_{H}\\
% =&\sum_{k=1}^{i} (\delta u_n^{k-1}, \delta w_n^{k-1})_{H}-(\delta u_n^0,\delta w_n^0)_{H}+(\delta u_n^i,\delta w_n^i)_{H}-\sum_{k=1}^i  (\delta u_n^{k-1}, \delta w_n^{k})_{H}\\
=(\delta u_n^i,\delta w_n^i)_{H}-(\delta u_n^0,\delta w_n^0)_{H}-\sum_{k=1}^{i}\tau_n (\delta u_n^{k-1},\delta^2 w_n^{k})_{H},
\end{equation}
from which we deduce that for every $\e>0$
\begin{align}
\left|\sum_{k=1}^i \tau_n (\delta^2 u_n^k, \delta w_n^k)_{H}\right|
\leq & \|\delta u_n^i\|_H\|\delta w_n^i\|_H+\|u^1\|_H\|\dot w(0)\|_H+\sum_{k=1}^i \tau_n\|\delta u_n^{k-1}\|_H\|\delta^2 w_n^k\|_H\nonumber\\
\leq & \frac{1}{2\epsilon}\|\delta w_n^i\|^2_H+\frac{\epsilon}{2}\|\delta u_n^i\|^2_H+\|u^1\|_H\|\dot w(0)\|_H+\frac{1}{2}\sum_{k=1}^i \tau_n\|\delta u_n^{k-1}\|_H^2+\frac{1}{2}\sum_{k=1}^i \tau_n\|\delta^2 w_n^k\|_H^2\nonumber\\
%\leq &\frac{1+\e}{2\epsilon}\|\dot{w}\|^2_{L^{\infty}(0,T;H)}+\frac{\epsilon}{2}\|\delta u_n^i\|^2_H+\frac{1}{2}\sum_{k=1}^i \tau_n\|\delta u_n^{k}\|^2_H+\frac{1+T}{2}\|u^1\|^2_H+\frac{1}{2}\|\ddot{w}\|^2_{L^2(0,T;H)}\nonumber\\
\leq &\bar C_\e+\frac{\epsilon}{2}\|\delta u_n^i\|^2_H+\frac{1}{2}\sum_{k=1}^i \tau_n\|\delta u_n^{k}\|^2_H,
\end{align}
where $\bar C_\e$ is a positive constant depending on $\e$. We estimate from above the last term in right-hand side of~\eqref{grw} in the following way
\begin{equation}\label{6}
\begin{aligned}
\sum_{k=1}^i \tau_n b_n^k(\delta u_n^k,\delta w_n^k)&\leq \sum_{k=1}^i \tau_n(b_n^k(\delta u_n^k,\delta u_n^k))^{\frac{1}{2}}(b_n^k(\delta w_n^k,\delta w_n^k))^{\frac{1}{2}}
% \leq \frac{1}{2}\sum_{k=1}^i \tau_n b_n^k(\delta u_n^k,\delta u_n^k)+\frac{1}{2}\sum_{k=1}^i \tau_n b_n^k(\delta w_n^k,\delta w_n^k)
\\
&\leq \frac{1}{2}\sum_{k=1}^i \tau_n b_n^k(\delta u_n^k,\delta u_n^k)+\frac{1}{2}\norm{\mathbb B}_\infty\norm{\Psi}_\infty^2\norm{\dot w}_{L^2(0,T;V_0)}^2.
% \|\Psi_n^k \E \delta u_n^k\|_H\|\mathbb{B}\Psi_n^k \E \delta w _n^k\|_H\leq \frac{\tilde{\epsilon}}{2}\sum_{k=1}^i\tau_n \|\Psi_n^k \E \delta u_n^k\|^2_H+\frac{1}{2\tilde{\epsilon}}\sum_{k=1}^i\tau_n \|\mathbb{B}\Psi_n^k \E \delta w_n^k\|^2_H\leq\nonumber\\
% \le&\frac{\tilde{\epsilon}}{2\lambda_2}\sum_{k=1}^i\tau_n b(\delta u_n^k,\delta u_n^k)+\frac{1}{2\tilde{\epsilon}}\|\mathbb{B}\|_\infty\|\Psi\|^2_\infty\sum_{k=1}^i\tau_n \| \E \delta w_n^k\|^2_H\leq\nonumber\\
% \le& \frac{\tilde{\epsilon}}{2\lambda_2}\sum_{k=1}^i\tau_n b(\delta u_n^k,\delta u_n^k)+\frac{C_9}{2\tilde{\epsilon}}\| \E \dot{w}\|^2_{L^2(0,T;H)}.
\end{aligned}
\end{equation}
By considering~\eqref{grw}--\eqref{6} and using~\eqref{CB2} we obtain
\begin{equation*}
\left(\frac{1-\epsilon}{2}\right)\|\delta u_n^i\|^2_H+\frac{\lambda_1-\epsilon C_{tr}^2}{2}\norm{\E u_n^i}_H^2+\frac{1}{2}\sum_{k=1}^i\tau_n b^k_n(\delta u_n^k,\delta u_n^k)\leq \hat C_\epsilon+\hat C\sum_{k=1}^i\tau_n\left(\|\delta u_n^k\|^2_H+ \|\E u_n^k\|^2_H\right)
\end{equation*} 
for two positive constants $\hat C_\epsilon$ and $\hat C$, with $\hat C_\e$ depending on $\epsilon$. We can now choose $\epsilon<\frac{1}{2}\min\left\{1,\frac{\lambda_1}{C_{tr}^2}\right\}$ to derive the following estimate
\begin{align}\label{last}
\frac{1}{4}\|\delta u_n^i\|^2_H+\frac{1}{4}\norm{\E u_n^i}_H^2+\frac{1}{2}\sum_{k=1}^i \tau_n b^k_n(\delta u_n^k,\delta u_n^k) \leq C_1+C_2\sum_{k=1}^i\tau_n\left(\|\delta u_n^k\|^2_H+ \|\E u_n^k\|^2_H\right),
% K_1+2K_2\sum_{k=1}^i\tau_n\|\delta u_n^k\|^2_H+2(K_3+\epsilon C_5)\sum_{k=1}^i \tau_n \|e u_n^k\|^2_H\leq\\
% \leq & 2K_1+2K_2\sum_{k=1}^i\tau_n\|\delta u_n^k\|^2_H+\frac{2(K_3+\epsilon C_5)}{\lambda_2}\sum_{k=1}^i \tau_n a( u_n^k, u_n^k)\leq\\
% \leq & 2K_1+\max\left(2K_2,\frac{2(K_3+\epsilon C_5)}{\lambda_2}\right)\sum_{k=1}^i \tau_n (\|\delta u_n^k\|^2_H+a(u_n^k,u_n^k)),
\end{align}
where $C_1$ and $C_2$ are two positive constants depending only on $u^0$, $u^1$, $f$, $g$, and $w$.
% \begin{equation}
% \|\delta u_n^i\|^2_H+a(u_n^i,u_n^i)\leq K'_1+K'_2\sum_{k=1}^i \tau_n (\|\delta u_n^k\|^2_H+a(u_n^k,u_n^k)).
% \end{equation}
Thanks to a discrete version of Gronwall's lemma (see, e.g.,~\cite[Lemma 3.2.4]{AGS}) we deduce the existence of a constant $C_3>0$, independent of $i$ and $n$, such that
\begin{equation*}
\|\delta u_n^i\|_H+\norm{\E u_n^i}_H\leq C_3 \quad \text{for every $i=1,\dots,n$ and for every $n\in\mathbb N$}.
\end{equation*}
By combining this last estimate with~\eqref{last} and~\eqref{CB2} we finally get~\eqref{eq:est} and we conclude.
% \begin{equation}
% \max_{i=1,..,n}\|\delta u_n^i\|^2_H+\max_{i=1,..,n}\norm{\E u_n^i}_H^2+\sum_{k=1}^n \tau_n b(\delta u_n^k,\delta u_n^k)\leq \tilde{K}'.
% \end{equation}
\end{proof}

We now want to pass to the limit into the discrete equation~\eqref{unk} to obtain a weak solution to~\eqref{visco_elasto_system}--\eqref{boundary}. We start by defining the following approximating sequences of our limit solution
\begin{align*}
&u_n(t):=u_n^k+(t-k\tau_n)\delta u_n^k, &&\tilde{u}_n(t):=\delta u_n^k+(t-k\tau_n)\delta^2 u_n^k &&t\in [(k-1)\tau_n,k\tau_n],\, k=1,\dots,n,\\
&u^+_n(t):=u_n^k, &&\tilde{u}^+_n(t):=\delta u_n^k &&t\in ((k-1)\tau_n,k\tau_n],\, k=1,\dots,n,\\
&u^-_n(t):=u_n^{k-1}, &&\tilde{u}^-_n(t):=\delta u_n^{k-1}&&t\in [(k-1)\tau_n,k\tau_n),\, k=1,\dots,n.
\end{align*}
Notice that $u_n\in H^1(0,T;H)$ with $\dot u_n(t)=\delta u_n^k=\tilde{u}^+_n(t)$ for $t\in ((k-1)\tau_n,k\tau_n)$ and $k=1,\dots,n$. Let us approximate $\Psi$ and $w$ by 
\begin{align*}
&\Psi^+_n(t):=\Psi_n^k, && w^+_n(t):=w_n^k && t\in ((k-1)\tau_n,k\tau_n],\, k=1,\dots,n,\\
&\Psi^-_n(t):=\Psi_n^{k-1}, && w^-_n(t):=w_n^{k-1} && t\in [(k-1)\tau_n,k\tau_n),\, k=1,\dots,n.
\end{align*}

\begin{lemma}\label{lem:conv}
There exists a function $u\in\mathcal W$, with $u-w\in\mathcal V^D$, such that, up to a not relabeled subsequence
\begin{align}
&u_n \xrightharpoonup[n\to \infty]{H^1(0,T;H)}u,\quad u^\pm_n \xrightharpoonup[n\to \infty]{L^2(0,T;V_T)}u,\quad \tilde{u}^\pm_n \xrightharpoonup[n\to \infty]{L^2(0,T;H)}\dot{u},\label{pf1}\\
&\nabla \Psi^\pm_n\odot \tilde{u}^\pm_n\xrightharpoonup[n\to \infty]{L^2(0,T;H)}\nabla\Psi\odot \dot u,\quad \E (\Psi^\pm_n\tilde{u}^\pm_n)\xrightharpoonup[n\to \infty]{L^2(0,T;H)}\E (\Psi\dot u).\label{pf2}
\end{align}
\end{lemma}

\begin{proof}
Thanks to Lemma~\ref{lem:est} the sequences $\{u_n\}_n\subset H^1(0,T;H)\cap L^\infty(0,T;V_T)$, $\{u_n^\pm\}_n\subset L^\infty(0,T;V_T)$, and $\{\tilde u_n^\pm\}_n\subset L^\infty (0,T;H)$ are uniformly bounded.
% Let's prove it:
% \begin{align}
% \|u^+_n\|^2_{L^2(0,T;V_T)}=&\int_0^T \|u^+_n(t)\|^2_{V_T} \,\de t=\sum_{k=1}^n\int_{(k-1)\tau_n}^{k\tau_n} \|u^k_n\|^2_{V_T} \,\de t\leq K \sum_{k=1}^n \tau_n(\|u_n^k\|^2_H+\|\E u_n^k\|^2_H)\leq\nonumber\\
% \leq & K \sum_{k=1}^n \tau_n\left(\left(\lambda_1+\sum_{i=1}^k\tau_n\|\delta u_n^i\|_H\right)^2+\frac{1}{\lambda_1}\int_{\Omega}a(u_n^k,u_n^k)\,\de x\right)\leq K\left(c^2_2+\frac{c_3}{\lambda_1}\right)\sum_{k=1}^n \tau_n\leq c_4T,\nonumber\\
% \|\tilde{u}^+_n\|^2_{L^2(0,T;H)}=&\int_0^T \|\tilde{u}^+_n(t)\|^2_{H} \,\de t=\sum_{k=1}^n\int_{(k-1)\tau_n}^{k\tau_n} \|\delta u^k_n\|^2_{H} \,\de t\leq \tilde{K}' \sum_{k=1}^n \tau_n\leq\tilde{K}'T,\label{der}\\
% \|u_n\|^2_{L^2(0,T;H)}=&\int_0^T \|u_n(t)\|^2_{H} \,\de t\leq 2\sum_{k=1}^n\int_{(k-1)\tau_n}^{k\tau_n}\left(\|u_n^k\|^2_H+(t-k\tau_n)^2 \|\delta u^k_n\|^2_{H} \,\de t\right)\leq\nonumber\\
% \leq & 2(\lambda_1+\tilde{K}'T)^2\sum_{k=1}^n\tau_n+\frac{2}{3}\tilde{K}'\sum_{k=1}^n\tau^3_n\leq 2(\lambda_1+\tilde{K}'T)^2T+\frac{2}{3}\tilde{K}'T\tau^2_n\leq c_5T^3,\nonumber
% \end{align}   
By Banach-Alaoglu's theorem there exist $u\in H^1(0,T;H)$ and $v\in L^2(0,T;V_T)$ such that, up to a not relabeled subsequence
\begin{equation*}
u_n \xrightharpoonup[n\to \infty]{L^2(0,T;V_T)}u,\quad \dot{u}_n \xrightharpoonup[n\to \infty]{L^2(0,T;H)}\dot{u},\quad u^+_n \xrightharpoonup[n\to \infty]{L^2(0,T;V_T)}v.
\end{equation*}
Since there exists a constant $C>0$ such that
\begin{align*}
\|u_n-u^+_n\|_{L^\infty(0,T;H)}\leq C \tau_n\xrightarrow[n\to\infty]{}0,
\end{align*}
we can conclude that $u=v$. Moreover, given that $ u^-_n(t)=u^+_n(t-\tau_n)$ for $t\in(\tau_n,T)$, $\tilde{u}^+_n(t)=\dot{u}_n(t)$ for a.e. $t\in(0,T)$, and $\tilde u^-_n(t)=\tilde u^+_n(t-\tau_n)$ for $t\in(\tau_n,T)$, we deduce 
\begin{equation*}
u^-_n \xrightharpoonup[n\to \infty]{L^2(0,T;V_T)} u,\quad \tilde{u}^\pm_n \xrightharpoonup[n\to \infty]{L^2(0,T;H)}\dot{u}.
\end{equation*}

By~\eqref{eq:est} we derive that the sequences $\{\E (\Psi^+_n\tilde{u}^+_n)\}_n\subset L^2(0,T;H)$ and $\{\nabla \Psi^+_n\odot \tilde{u}^+_n\}_n\subset L^2(0,T;H)$ are uniformly bounded. Indeed there exists a constant $C>0$ independent of $n$ such that
\begin{align*}
\|\nabla\Psi^+_n\odot \tilde{u}^+_n\|^2_{L^2(0,T;H)}&
% =\int_0^T\|\nabla\Psi^+_n\odot \tilde{u}^+_n(t)\|^2_H \,\de t
=\sum_{k=1}^n\int_{(k-1)\tau_n}^{k\tau_n}\|\nabla\Psi^k_n\odot \delta u^k_n\|^2_H \,\de t
% =\sum_{k=1}^n\tau_n\|\nabla\Psi^k_n\odot \delta u^k_n\|^2_H\\
% =\sum_{k=1}^n\tau_n \int_{\Omega}|\nabla \Psi^k_n|^2|\delta u^k_n|^2\,\de x
\leq \|\nabla \Psi\|^2_\infty\sum_{k=1}^n \tau_n \|\delta u^k_n\|^2_{H}\leq C,\\
\|\E (\Psi^+_n\tilde{u}^+_n)\|^2_{L^2(0,T;H)}&
% =\int_0^T\|\E (\Psi^+_n\tilde{u}^+_n)(t)\|^2_H \,\de t
=\sum_{k=1}^n\int_{(k-1)\tau_n}^{k\tau_n}\|\E (\Psi^k_n\delta u^k_n)\|^2_H \,\de t
% =\sum_{k=1}^n\tau_n\|\E (\Psi^k_n\delta u^k_n)\|^2_H=\\
=\sum_{k=1}^n\tau_n \|\Psi^k_n \E \delta u^k_n+\nabla\Psi^k_n\odot\delta u^k_n\|^2_H\\
&\leq 2\sum_{k=1}^n\tau_n\|\Psi^k_n \E \delta u^k_n\|^2_H+ 2\sum_{k=1}^n\tau_n\|\nabla\Psi^k_n\odot\delta u^k_n\|^2_H\leq C.
% \\
% \leq & 2\tilde{K}'T\|\nabla\Psi\|^2_\infty+\frac{2}{\lambda_1}\sum_{k=1}^n\tau_n\int_{\Omega}\lambda_2|\Psi^k_n|^2|\E \delta u^k_n|^2\,\de x\leq \\
% \leq & c'_1+\frac{2}{\lambda_1}\sum_{k=1}^n \tau_n\int_{\Omega}|\Psi^k_n|^2\mathbb{B}\E \delta u^k_n\cdot \E \delta u^k_n\,\de x = c'_1+\frac{2}{\lambda_1}\sum_{k=1}^n \tau_n b(\delta u^k_n,\delta u^k_n)\leq c'_1+\frac{2}{\lambda_1}\tilde{K}',
\end{align*}
Therefore, there exists $w_1,w_2\in L^2(0,T;H)$ such that, up to a further not relabeled subsequence 
\begin{equation*}
\nabla \Psi^+_n\odot \tilde{u}^+_n\xrightharpoonup[n\to \infty]{L^2(0,T;H)}w_1,\quad \E (\Psi^+_n\tilde{u}^+_n)\xrightharpoonup[n\to \infty]{L^2(0,T;H)}w_2.
\end{equation*}
We want to identify the limit functions $w_1$ and $w_2$. Consider $\varphi\in L^2(0,T;H)$, then 
\begin{equation*}
\int_0^T(\nabla \Psi^+_n\odot \tilde{u}^+_n,\varphi)_{H} \,\de t=\frac{1}{2}\int_0^T(\tilde{u}^+_n,\varphi\nabla \Psi^+_n)_{H} \,\de t+\frac{1}{2}\int_0^T(\tilde{u}^+_n,\varphi^T\nabla \Psi^+_n)_{H} \,\de t=\int_0^T(\tilde{u}^+_n,\varphi^{sym}\nabla\Psi^+_n)_{H} \,\de t,
\end{equation*} 
where $\varphi^{sym}:=\frac{\varphi+\varphi^T}{2}$. Since $\tilde{u}^+_n\xrightharpoonup[n\to \infty]{L^2(0,T;H)}\dot{u}$ and $\varphi^{sym}\nabla\Psi^+_n \xrightarrow[n\to \infty]{L^2(0,T;H)}\varphi^{sym}\nabla \Psi$ by dominated convergence theorem, we obtain
\begin{equation*}
\int_0^T(\nabla \Psi^+_n\odot \tilde{u}^+_n,\varphi)_{H} \,\de t\xrightarrow[n\to \infty]{}\int_0^T(\dot{u},\varphi^{sym}\nabla\Psi)_{H} \,\de t=\int_0^T(\nabla\Psi\odot \dot{u},\varphi)_{H} \,\de t,
\end{equation*}
and so $w_1=\nabla\Psi \odot \dot{u}$. Moreover for $\phi\in L^2(0,T;H)$ we have
\begin{equation*}
\int_0^T(\Psi^+_n\tilde{u}^+_n,\phi)_{H} \,\de t=\int_0^T(\tilde{u}^+_n,\phi \Psi^+_n)_{H} \,\de t\xrightarrow[n\to\infty]{}\int_0^T(\dot{u},\Psi\phi)_{H} \,\de t=\int_0^T(\Psi\dot{u},\phi)_{H} \,\de t,
\end{equation*} 
thanks to $\tilde{u}^+_n\xrightharpoonup[n\to \infty]{L^2(0,T;H)}\dot{u}$ and $\Psi^+_n\phi\xrightarrow[n\to\infty]{L^2(0,T;H)}\Psi\phi$, again implied by dominated convergence theorem. Therefore $\Psi^+_n\tilde{u}^+_n\xrightharpoonup[n\to\infty]{L^2(0,T;H)}\Psi \dot{u}$, from which
% $\Psi^+_n\tilde{u}^+_n\xrightarrow[n\to\infty]{\mathcal{D}'(0,T;H)}\Psi \dot{u}$ and at the end
$\E (\Psi^+_n\tilde{u}^+_n)\xrightarrow[n\to\infty]{\mathcal{D}'(0,T;H)}\E (\Psi \dot{u})$, that gives $w_2=\E (\Psi\dot{u})$. In particular we have $\Psi\dot u\in L^2(0,T;V_T)$. By arguing in a similar way we also obtain
\begin{equation*}
\nabla \Psi^-_n\odot \tilde{u}^-_n\xrightharpoonup[n\to \infty]{L^2(0,T;H)}\nabla\Psi \odot \dot{u},\quad \E (\Psi^-_n\tilde{u}^-_n)\xrightharpoonup[n\to \infty]{L^2(0,T;H)}\E (\Psi\dot{u}).
\end{equation*}

Let us check that $u\in\mathcal W$. To this aim, let us consider the following set
\begin{equation*}
F:=\{v\in L^2(0,T;V_T): v(t)\in V_t\text{ for a.e. $t\in (0,T)$}\}\subset L^2(0,T;V_T).
\end{equation*}
We have that $F$ is a (strong) closed convex subset of $L^2(0,T;V_T)$, and so by Hahn-Banach's theorem the set $F$ is weakly closed. Notice that $\{u^-_n\}_n,\{\Psi^-_n\tilde{u}^-_n\}_n\subset F$, indeed
\begin{align*}
&u^-_n(t)=u^{k-1}_n\in V_{(k-1)\tau_n}\subset V_t\quad\text{for $t\in [(k-1)\tau_n,k\tau_n)$, $k=1,\dots,n$},\\
&\Psi_n^-(t)\tilde u^-_n(t)=\Psi_n^{k-1}\delta u^{k-1}_n\in V_{(k-1)\tau_n}\subseteq V_t\quad\text{for $t\in [(k-1)\tau_n,k\tau_n)$, $k=1,\dots,n$}.
\end{align*}
Since $u^-_n\xrightharpoonup[n\to\infty]{L^2(0,T;V_T)} u$ and $\Psi_n^-\tilde u^-_n\xrightharpoonup[n\to\infty]{L^2(0,T;V_T)}\Psi\dot u$, we conclude that $u,\Psi \dot u\in F$. Finally, to show that $u-w\in \mathcal V^D$ we observe
\begin{align*}
u_n^-(t)-w_n^-(t)=u^{k-1}_n-w^{k-1}_n\in V_n^{k-1}\subseteq V_t^D\quad \text{ for $t\in[(k-1)\tau_n,k\tau_n)$, $k=1,\dots,n$}.
\end{align*}
Therefore $\{u_n^--w_n^-\}_n\subset \{v\in L^2(0,T;V_T): v(t)\in V_t^D\text{ for a.e. $t\in (0,T)$}\}$, which is a (strong) closed convex subset of $L^2(0,T;V_T)$, and so it is weakly closed. Since $u_n^-\xrightharpoonup[n\to\infty]{L^2(0,T;V_T)}u$ and $w_n^-\xrightarrow[n\to\infty]{L^2(0,T;V_0)}w$, we get
% \begin{equation*}
% u_n^--w_n^-\xrightharpoonup[n\to\infty]{L^2(0,T;V_T)}u-w.
% \end{equation*}
that $u(t)-w(t)\in V_t^D$ for a.e. $t\in(0,T)$, which implies $u-w\in\mathcal V^D$.
\end{proof}

We now use Lemma~\ref{lem:conv} to pass to the limit in the discrete equation~\eqref{unk}.

\begin{lemma}\label{lem:exis}
The limit function $u\in\mathcal W$ of Lemma~\ref{lem:conv} is a weak solution to~\eqref{visco_elasto_system}--\eqref{boundary}.
\end{lemma}

\begin{proof}
We only need to prove that $u\in\mathcal W$ satisfies~\eqref{weakform}. We fix $n\in\mathbb N$, $\varphi\in C_c^1(0,T;V_T)$ such that $\varphi(t)\in V_t^D$ for every $t\in(0,T)$, and we consider 
\begin{align*}
&\varphi_n^k:=\varphi(k\tau_n)\quad\text{for $k=0,\dots,n$,}\quad\delta \varphi_n^k:=\frac{\varphi_n^k-\varphi_n^{k-1}}{\tau_n}\quad \text{for }k=1,\dots,n,
\end{align*}
and the approximating sequences
\begin{align*}
&\varphi^+_n(t):=\varphi_n^k, & & \tilde \varphi^+_n(t):=\delta\varphi_n^k & & t\in ((k-1)\tau_n,k\tau_n],\, k=1,\dots,n.
%&\varphi^-_n(t):=\varphi_n^{k-1}, && \tilde \varphi^-_n(t):=\delta \varphi_n^{k-1} && t\in [(k-1)\tau_n,k\tau_n),\, k=1,\dots,n.
\end{align*}
If we use $\tau_n\varphi_n^k\in V_n^k$ as test function in~\eqref{unk}, after summing over $k=1,...,n$, we get
\begin{equation}\label{limit}
\begin{aligned}
\sum_{k=1}^n\tau_n(\delta^2u_n^k,\varphi^k_n)_H+\sum_{k=1}^n\tau_n(\mathbb C\E u_n^k,\E \varphi^k_n)_{H}&+\sum_{k=1}^n\tau_n(\mathbb B\Psi_n^k\E \delta u_n^k,\Psi_n^k \E \varphi^k_n)_{H}\\
&=\sum_{k=1}^n\tau_n(f_n^k,\varphi^k_n)_H+\sum_{k=1}^n\tau_n(g_n^k,\varphi^k_n)_{H_N}.
\end{aligned}
\end{equation}
By these identities
\begin{align*}
\sum_{k=1}^n \tau_n(\delta^2 u^k_n,\varphi^k_n)_H
% =&\sum_{k=1}^n (\delta u^k_n,\varphi^k_n)_H-\sum_{k=1}^n (\delta^2 u^{k-1}_n,\varphi^k_n)_H=\sum_{k=2}^{n+1} (\delta u^{k-1}_n,\varphi^{k-1}_n)_H-\sum_{k=1}^n (\delta^2 u^{k-1}_n,\varphi^k_n)_H=\\
% =&\sum_{k=1}^n (\delta u^{k-1}_n,\varphi^{k-1}_n)_H-(\delta u^0_n,\varphi^0_n)_H+(\delta u^n_n,\varphi^n_n)_H-\sum_{k=1}^n (\delta^2 u^{k-1}_n,\varphi^k_n)_H=\\
=& -\sum_{k=1}^n\tau_n(\delta u^{k-1}_n,\delta \varphi^k_n)_H=-\int_0^T(\tilde{u}^-_n(t),\tilde{\varphi}^+_n(t))_H \,\de t,
\end{align*}
from~\eqref{limit} we deduce
\begin{equation}\label{eqapp}
\begin{aligned}
-\int_0^T(\tilde{u}^-_n,\tilde{\varphi}^+_n)_H \,\de t+&\int_0^T(\mathbb{C}\E u^+_n,\E \varphi^+_n)_{H} \,\de t-\int_0^T (\mathbb{B}\nabla \Psi^+_n\odot \tilde{u}^+_n,\E \varphi^+_n)_{H} \,\de t\\
+&\int_0^T (\mathbb{B}\E (\Psi^+_n\tilde{u}^+_n),\E \varphi^+_n)_{H} \,\de t=\int_0^T(f^+_n,\varphi^+_n)_H \,\de t+\int_0^T(g^+_n,\varphi^+_n)_{H_N} \,\de t.
\end{aligned}
\end{equation}
Thanks to~\eqref{pf1},~\eqref{pf2}, and the following convergences 
\begin{equation*}
\varphi^+_n\xrightarrow[n\to\infty]{L^2(0,T;V_T)}\varphi, \quad \tilde{\varphi}^+_n\xrightarrow[n\to\infty]{L^2(0,T;H)}\dot{\varphi},\quad f^+_n\xrightarrow[n\to\infty]{L^2(0,T;H)}f ,\quad g^+_n\xrightarrow[n\to\infty]{L^2(0,T;H_N)}g,
\end{equation*}
we can pass to the limit in~\eqref{eqapp}, and we get that $u\in\mathcal W$ satisfies~\eqref{weakform} for every $\varphi\in C_c^1(0,T;V_T)$ such that $\varphi(t)\in V_t^D$ for every $t\in(0,T)$. Finally, by using a density argument (see~\cite[Remark 2.9]{DMT2}), we conclude that $u\in\mathcal W$ is a weak solution to~\eqref{visco_elasto_system}--\eqref{boundary}.
\end{proof}

\section{Initial Conditions and Energy--Dissipation Inequality}\label{visc_sec3}

To complete our existence result, it remains to prove that the function $u\in\mathcal W$ given by Lemma~\ref{lem:exis} satisfies the initial conditions~\eqref{initials} in the sense of~\eqref{icd}. Let us start by showing that the second distributional derivative $\ddot{u}$ belongs to $L^2(0,T;(V_0^D)')$. If we consider the discrete equation~\eqref{unk}, for every $v\in V_0^D\subseteq V_n^k$, with $\|v\|_{V_0}\leq 1$, we have
\begin{align*}
|(\delta^2 u^k_n,v)_H|\leq & \|\mathbb{C}\|_\infty\|\E u^k_n\|_H+\|\mathbb{B}\|_\infty\|\Psi\|_\infty\|\Psi^k_n \E \delta u^k_n\|_H+\| f^k_n\|_H+C_{tr}\| g^k_n\|_{H_N}.
% \le& \|\mathbb{C}\|_\infty\|\E u^k_n\|_H\| v\|_{V_0}+\|\mathbb{B}\|_\infty\|\Psi\|_{L^{\infty}_{t,x}}\|\Psi^k_n \E \delta u^k_n\|_H\|v\|_{V_0}+\| f^k_n\|_H\| v\|_{V_0}+\| g^k_n\|_H\| v\|_{V_0}
% |(\delta^2 u^k_n,v)_H|\leq & \|\mathbb{C}\|_\infty\|\E u^k_n\|_H\|\E v\|_H+\|\mathbb{B}\|_\infty\|\Psi^k_n \E \delta u^k_n\|_H\|\Psi^k_n \E v\|_H+\| f^k_n\|_H\| v\|_H+\| g^k_n\|_H\| v\|_H\leq\\
% \le& \|\mathbb{C}\|_\infty\|\E u^k_n\|_H\| v\|_{V_0}+\|\mathbb{B}\|_\infty\|\Psi\|_{L^{\infty}_{t,x}}\|\Psi^k_n \E \delta u^k_n\|_H\|v\|_{V_0}+\| f^k_n\|_H\| v\|_{V_0}+\| g^k_n\|_H\| v\|_{V_0}
\end{align*}
Therefore, taking the supremum over $v\in V_0^D$ with $\norm{v}_{V_0}\leq1$, we obtain the existence of a positive constant $C$ such that 
\begin{equation*}
\|\delta^2 u^k_n\|_{(V_0^D)'}^2\leq C(\|\E u^k_n\|_H^2+\|\Psi^k_n \E \delta u^k_n\|_H^2+\|f^k_n\|_H^2+\|g^k_n\|_{H_N}^2).
\end{equation*}
If we multiply this inequality by $\tau_n$ and we sum over $k=1,\dots,n$, we get
\begin{align}\label{d2}
\sum_{k=1}^n\tau_n\|\delta^2 u^k_n\|^2_{(V^D_0)'}\leq C\left(\sum_{k=1}^n\tau_n\|\E u^k_n\|_H^2+\sum_{k=1}^n\tau_n\|\Psi^k_n \E \delta u^k_n\|_H^2+\|f\|_{L^2(0,T;H)}^2+\|g\|_{L^2(0,T;H_N)}^2\right).
% \sum_{k=1}^n\tau_n\|\delta^2 u^k_n\|^2_{V'_0}\leq &C_1\frac{\tilde{K}'^2}{\lambda_1}T+l_2 \sum_{k=1}^n\tau_n\|\Psi^k_n \E \delta u^k_n\|^2_H+2\sum_{k=1}^n\tau_n\|f^k_n\|^2_H+2\sum_{k=1}^n\tau_n\|g^k_n\|^2_H\leq\nonumber\\
% \leq & l_1\frac{\tilde{K}'^2}{\nu^2_1}T+l_2\frac{\tilde{K}'^2}{\nu^2_2}T+2\|f\|^2_{L^2(0,T;H)}+2\|g\|^2_{L^2(0,T;H)}\leq \tilde{K}''.
\end{align}
Thanks to~\eqref{d2} and Lemma~\ref{lem:est} we conclude that $\sum_{k=1}^n\tau_n\|\delta^2 u^k_n\|^2_{(V_0^D)'}\leq\tilde C$ for every $n\in\mathbb N$ for a positive constant $\tilde C$ independent on $n\in\mathbb N$. In particular the sequence $\{\tilde u_n\}_n\subset H^1(0,T;(V_0^D)')$ is uniformly bounded (notice that $\dot{\tilde u}_n(t)=\delta^2 u_n^k$ for $t\in((k-1)\tau_n,k\tau_n)$ and $k=1,\dots,n$). Hence, 
up to extract a further (not relabeled) subsequence  from the one of Lemma~\ref{lem:conv}, we get
\begin{equation}\label{new}
\tilde{u}_n\xrightharpoonup[n\to\infty]{H^1(0,T;(V^D_0)')}w_3,
\end{equation}
and by using the following estimate
\begin{equation*}
\|\tilde{u}_n-\tilde{u}^+_n\|_{L^2(0,T;(V^D_0)')}\leq \tau_n\|\dot{\tilde{u}}_n\|_{L^2(0,T;(V^D_0)')}\leq \tilde C\tau_n\xrightarrow[n\to\infty]{}0
\end{equation*}
we conclude that $w_3=\dot{u}$. 

Let us recall the following result, whose proof can be found for example in~\cite{DL}.

\begin{lemma}\label{lem:wc}
Let $X,Y$ be two reflexive Banach spaces such that $X\hookrightarrow Y$ continuously. Then 
$$L^{\infty}(0,T;X)\cap C^0_w([0,T];Y)= C^0_w([0,T];X).$$
\end{lemma}

Since $H^1(0,T;(V^D_0)')\hookrightarrow C^0([0,T],(V^D_0)')$, by using Lemmas~\ref{lem:conv} and~\ref{lem:wc} we get that our weak solution $u\in \mathcal{W}$ satisfies
\begin{equation*}
 u\in C^0_w([0,T];V_T),\quad \dot{u}\in C^0_w([0,T];H),\quad \ddot{u} \in L^2(0,T;(V^D_0)').
\end{equation*} 
By~\eqref{pf1} and~\eqref{new} we hence obtain
\begin{equation}\label{weakco}
u_n(t)\xrightharpoonup[n\to\infty]{H}u(t),\quad \tilde{u}_n(t)\xrightharpoonup[n\to\infty]{(V^D_0)'}\dot{u}(t)\quad \text{for every } t\in [0,T],
\end{equation}
so that $u(0)=u^0$ and $\dot u(0)=u^1$, since $u_n(0)=u^0$ and $\tilde{u}_n(0)=u^1$.

To prove that 
\begin{equation*}
\lim_{h\to 0^+}\frac{1}{h}\int_0^h\left(\|u(t)-u^0\|^2_{V_t}+\|\dot{u}(t)-u^1\|^2_{H}\right) \,\de t=0
\end{equation*} 
we will actually show
\begin{equation*}
    \lim_{t\to 0^+}u(t)= u^0\text{ in $V_T$},\quad\lim_{t\to 0^+}\dot u(t)=u^1\text{ in $H$}.
\end{equation*}
This is a consequence of following energy--dissipation inequality which holds for the weak solution $u\in\mathcal W$ of Lemma~\ref{lem:exis}. Let us define the total energy as
\begin{equation*}
    \mathcal{E}(t):=\frac{1}{2}\norm{\dot u(t)}_H^2+\frac{1}{2}(\mathbb C\E u(t),\E u(t))_H\quad t\in[0,T].
\end{equation*}
Notice that $\mathcal E(t)$ is well defined for every $t\in[0,T]$ since $u\in C_w^0([0,T];V_T)$ and $\dot u\in C_w^0([0,T];H)$, and that $\mathcal E(0)=\frac{1}{2}\norm{ u^1}_H^2+\frac{1}{2}(\mathbb C\E u^0,\E u^0)_H$.

\begin{theorem}
The weak solution $u\in\mathcal W$ to~\eqref{visco_elasto_system}--\eqref{boundary}, given by Lemma~\ref{lem:exis}, satisfies for every $t\in[0,T]$ the following energy--dissipation inequality 
\begin{equation}\label{eq:enin}
\mathcal{E}(t)+\int_0^{t} (\mathbb{B}\Psi \E \dot{u},\Psi \E \dot{u} )_H\,\de s \leq \mathcal{E}(0)+\mathcal W_{tot}(t),
\end{equation}
where $\Psi \E \dot u$ is the function defined in~\eqref{Psiedotu} and $\mathcal W_{tot}(t)$ is the total work on the solution $u$ at time $t\in[0,T]$, which is given by
\begin{equation}\label{totalwork}
\begin{aligned}
\mathcal W_{tot}(t):&=\int_0^{t}\left[(f,\dot{ u}-\dot w)_H +(\mathbb{C}\E u,\E \dot{w})_H +(\mathbb{B}\Psi \E \dot{u},\Psi \E \dot{ w})_H -(\dot{u},\ddot{ w})_H-(\dot g,u-w)_{H_N} \right]\de s\\
&\quad+(\dot{u}(t),\dot{w}(t))_H +(g(t),u(t)-w(t))_{H_N} -(u^1,\dot{w}(0))_H -(g(0),u^0-w(0))_{H_N}.
\end{aligned}
\end{equation}
\end{theorem}

\begin{remark}
From the classical point of view, the total work on the solution $u$ at time $t\in[0,T]$ is given by
\begin{equation}\label{totalwork2}
    \mathcal W_{tot}(t):=\mathcal W_{load}(t)+\mathcal W_{bdry}(t),
\end{equation}
where $\mathcal W_{load}(t)$ is the work on the solution $u$ at time $t\in[0,T]$ due to the loading term, which is defined as
\begin{equation*}
   \mathcal W_{load}(t):=\int_0^t(f(s), \dot u(s))_H\,\de s,
\end{equation*}
and $\mathcal W_{bdry}(t)$ is the work on the solution $u$ at time $t\in[0,T]$ due to the varying boundary conditions, which one expects to be equal to
\begin{align*}
   \mathcal W_{bdry}(t):=\int_0^t(g(s),\dot u(s))_{H_N}\,\de s+\int_0^t((\mathbb C\E u(s)+\Psi^2(s)\mathbb B\E \dot u(s))\nu,\dot w(s))_{H_D}\,\de s,
\end{align*}
being $H_D:=L^2(\partial_D\Omega;\R^d)$. Unfortunately, $\mathcal W_{bdry}(t)$ is not well defined under our assumptions on $u$. Notice that when $\Psi\equiv 1$ on a neighborhood $U$ of the closure of $\partial_N\Omega$, then every weak solution $u$ to~\eqref{visco_elasto_system}--\eqref{boundary} satisfies $u\in H^1(0,T;H^1((\Omega\cap U)\setminus\Gamma;\R^d))$, which gives that $u\in H^1(0,T;H_N)$ by our assumptions on $\Gamma$. Hence the first term of $\mathcal W_{bdry}(t)$ makes sense and satisfies
\begin{align*}
    \int_0^t(g(s),\dot u(s))_{H_N}\,\de s=(g(t), u(t))_{H_N}-(g(0), u(0))_{H_N}- \int_0^t(\dot g(s), u(s))_{H_N}\,\de s.
\end{align*}
The term involving the Dirichlet datum $w$ is more difficult to handle since the trace of $(\mathbb C\E u+\Psi^2\mathbb B\E \dot u)\nu$ on $\partial_D\Omega$ is not well defined even when $\Psi\equiv 1$ on a neighborhood of the closure of $\partial_D\Omega$. If we assume that $u\in H^1(0,T;H^2(\Omega\setminus\Gamma;\R^d))\cap H^2(0,T;L^2(\Omega;\R^d))$ and that $\Gamma$ is a smooth manifold, then we can integrate by part equation~\eqref{weakform} to deduce that $u$ satisfies~\eqref{visco_elasto_system}. In this case, $(\mathbb C\E u+\Psi^2\mathbb B\E \dot u)\nu\in L^2(0,T;H_D)$ and by using~\eqref{visco_elasto_system}, together with the divergence theorem and the integration by parts formula, we deduce
\begin{align*}
&\int_0^t((\mathbb C\E u(s)+\Psi^2(s)\mathbb B\E \dot u(s))\nu,\dot w(s))_{H_D}\,\de s\\
&=\int_0^t\left[(\div(\mathbb C\E u(s)+\Psi^2(s)\mathbb B\E \dot u(s)),\dot w(s))_H+(\mathbb C\E u(s)+\Psi^2(s)\mathbb B\E \dot u(s),\E \dot w(s))_H-(g(s),\dot w(s))_{H_N}\right]\de s\\
&=\int_0^t\left[(\ddot u(s),\dot w(s))_H-(f(s),\dot w(s))_H+(\mathbb C\E u(s)+\Psi^2(s)\mathbb B\E \dot u(s),\E \dot w(s))_H-(g(s),\dot w(s))_{H_N}\right]\de s\\
&=\int_0^{t}\left[(\mathbb{C}\E u(s),\E \dot{w}(s))_H+(\mathbb{B}\Psi (s)\E \dot{u}(s),\Psi(s) \E \dot{ w}(s))_H-(f(s),\dot w(s))_H\right]\de s\\
&\quad+\int_0^t\left[(\dot g(s),w(s))_{H_N}-(\dot{u}(s),\ddot{ w}(s))_H\right]\de s-(g(t),w(t))_{H_N}+(\dot{u}(t),\dot{w}(t))_H +(g(0),w(0))_{H_N}-(u^1,\dot{w}(0))_H .
\end{align*}
Hence, the definition of total work given in~\eqref{totalwork} is coherent with the classical one~\eqref{totalwork2}. Notice that if $u$ is the solution to~\eqref{visco_elasto_system}--\eqref{boundary} given by Lemma~\ref{lem:exis}, then~\eqref{totalwork} is well defined for every $t\in[0,T]$, since $g\in C^0([0,T];H_N)$, $\dot w\in C^0([0,T];H)$, $u\in C_w^0([0,T];V_T)$, and $\dot u\in C_w^0([0,T];H)$. In particular, the function $t\mapsto \mathcal W_{tot}(t)$ from $[0,T]$ to $\R$ is continuous.
\end{remark}

\begin{proof}
% By~\eqref{grw} we know that
% \begin{align}\label{eq:disenin}
% \frac{1}{2}\|\delta u_n^i\|^2_H+\frac{1}{2}a(u_n^i,u_n^i)+\sum_{k=1}^i\tau_n b_n^k(\delta u_n^{k},\delta u_n^{k})\leq \mathcal{E}(0)+\sum_{k=1}^i\tau_n L^k_n,
% \end{align}
% for every $i=1,\dots,n$, where $L_n^k$ is the quantity defined in~\eqref{eq:Lnk}.
% % where, for definition, we set
% % \begin{align*}
% % \mathcal{L}^i_n(w):=&-\sum_{k=1}^i\tau_n (f^k_n,\delta w^k_n)_H-\sum_{k=1}^i\tau_n (g^k_n,\delta w^k_n)_H+\sum_{k=1}^i\tau_n (\mathbb{C} \E u^k_n,\E \delta w^k_n)_H+\\
% % +&\sum_{k=1}^i\tau_n (\delta^2 u^k_n,\delta w^k_n)_H+\sum_{k=1}^i\tau_n (\mathbb{B}\Psi^k_n \E \delta u^k_n,\Psi^k_n \E \delta w^k_n)_H,
% % \end{align*}
Fixed $t\in (0,T]$, for every $n\in\mathbb N$ there exists a unique $j\in\{1,\dots,n\}$ such that $t\in((j-1)\tau_n,j\tau_n]$. After setting $t_n:=j\tau_n$, we can rewrite~\eqref{grw} as
\begin{equation}\label{qf}
\frac{1}{2}\| \tilde{u}_n^+(t)\|^2_H+\frac{1}{2}(\mathbb C\E u_n^+(t),\E u_n^+(t))_H+\int_0^{t_n} (\mathbb B\Psi_n^+\E \tilde{u}_n^{+},\Psi_n^+\E \tilde{ u}_n^{+})_H\,\de s\leq \mathcal{E}(0)+\mathcal{W}^+_n(t),
\end{equation}
where
\begin{align*}
\mathcal{W}^+_n(t)&:=\int_0^{t_n}\left[(f^+_n,\tilde u_n^+-\tilde{ w}^+_n)_H+(\mathbb{C} \E u^+_n,\E \tilde{ w}^+_n)_H+(\mathbb{B}\Psi_n^+ \E \tilde u_n^+,\Psi_n^+ \E \tilde w_n^+)_H\right]\de s\\
&\quad+\int_0^{t_n}\left[({\tilde u}^+_n,\tilde w_n^+)_H+(g^+_n,\tilde u_n^+-\tilde{ w}^+_n)_{H_N}\right]\de s.
\end{align*}
Thanks to~\eqref{eq:est}, we have
\begin{align*}
&\|u_n(t)-u_n^+(t)\|_H=\|u_n^j+(t-j\tau_n)\delta u_n^j-u_n^j\|_H\leq \tau_n\|\delta u_n^j\|_H\leq C\tau_n\xrightarrow[n\to \infty]{}0,\\
&\|\tilde{u}_n(t)-\tilde{u}_n^+(t)\|^2_{(V_0^D)'}=\|\delta u_n^j+(t-j\tau_n)\delta^2 u_n^j-\delta u_n^j\|^2_{(V_0^D)'}\leq \tau^2_n\|\delta^2 u_n^j\|^2_{(V_0^D)'}\leq C\tau_n\xrightarrow[n\to \infty]{}0.
\end{align*}
The last convergences and~\eqref{weakco} imply 
\begin{equation*}
u_n^+(t)\xrightharpoonup[n\to\infty]{H}u(t), \qquad \tilde{u}_n^+(t)\xrightharpoonup[n\to\infty]{(V_0^D)'}\dot{u}(t),
\end{equation*}
and since $\|u_n^+(t)\|_{V_T}+\|\tilde{u}_n^+(t)\|_H\leq C$ for every $n\in\mathbb N$, we get
\begin{align}\label{newconv}
u_n^+(t)\xrightharpoonup[n\to \infty]{V_T}u(t), \qquad \tilde{u}_n^+(t)\xrightharpoonup[n\to \infty]{H}\dot{u}(t).
\end{align}
By the lower semicontinuity properties of $v\mapsto\norm{v}_H^2$ and  $v\mapsto (\mathbb{C}\E v,\E v)_H$, we conclude 
\begin{align}
&\| \dot{u}(t)\|^2_H\leq\liminf_{n\to\infty}\| \tilde{u}_n^+(t)\|^2_H,\label{se1}\\
&(\mathbb{C}\E u(t),\E u(t))_H \leq \liminf_{n\to\infty}(\mathbb C\E u_n^+(t),\E u_n^+(t))_H.
\end{align}
Thanks to Lemma~\ref{lem:conv} and~\eqref{Psiedotu}, we obtain
$$\Psi^+_n \E \tilde{u}^+_n=\E (\Psi_n^+\tilde u_n^+)-\nabla\Psi_n^+\odot\tilde u_n^+\xrightharpoonup[n\to\infty]{L^2(0,T;H)}\E (\Psi\dot{u})-\nabla \Psi\odot \dot{u}=\Psi \E \dot u,$$
so that
\begin{equation}
    \int_0^t (\mathbb B\Psi \E \dot{u},\Psi \E \dot{ u})_H\,\de s\leq\liminf_{n\to \infty}\int_0^t (\mathbb B\Psi_n^+\E \tilde{u}_n^{+},\Psi_n^+\E \tilde{ u}_n^{+})_H\,\de s\leq\liminf_{n\to \infty}\int_0^{t_n} (\mathbb B\Psi_n^+\E \tilde{u}_n^{+},\Psi_n^+\E \tilde{ u}_n^{+})_H\,\de s,
\end{equation}
since $t\leq t_n$ and $v\mapsto\int_0^t(\mathbb Bv,v)_H\,\de s$ is a non negative quadratic form on $L^2(0,T;H)$. 
Let us study the right--hand side of~\eqref{qf}. Given that we have
\begin{equation*}
\chi_{[0,t_n]}f^+_n\xrightarrow[n\to\infty]{L^2(0,T;H)}\chi_{[0,t]}f,\quad \tilde{u}^+_n-\tilde{w}^+_n\xrightharpoonup[n\to\infty]{L^2(0,T;H)}\dot{u}-\dot w,
\end{equation*}
we can deduce
\begin{align}
\int_0^{t_n}(f^+_n,\tilde{u}^+_n-\tilde{w}^+_n)_H\,\de s\xrightarrow[n\to\infty]{}\int_0^t(f,\dot{u}-\dot w)_H\,\de s.
\end{align}
In a similar way, we can prove 
\begin{align}
&\int_0^{t_n}(\mathbb{C}\E u^+_n,\E \tilde{ w}^+_n)_H\,\de s \xrightarrow[n\to\infty]{}\int_0^{t}(\mathbb{C}\E u,\E \dot{ w})_H\,\de s,\\
&\int_0^{t_n}(\mathbb{B}\Psi^+_n \E \tilde{u}^+_n,\Psi^+_n \E \tilde{ w}^+_n)_H\,\de s\xrightarrow[n\to\infty]{}\int_0^{t}(\mathbb{B}\Psi \E \dot u,\Psi \E \dot{ w})_H\,\de s,
\end{align} 
since the following convergences hold
\begin{align*}
&\chi_{[0,t_n]}\E \tilde{w}_n^+\xrightarrow[n\to\infty]{L^2(0,T;H)}\chi_{[0,t]}\E \dot{w},& &\mathbb{C}\E u^+_n\xrightharpoonup[n\to\infty]{L^2(0,T;H)}\mathbb{C}\E u,\\
&\chi_{[0,t_n]}\Psi^+_n \E \tilde{w}^+_n\xrightarrow[n\to\infty]{L^2(0,T;H)}\chi_{[0,t]}\Psi \E \dot{w}, & &\Psi^+_n \E \tilde{u}^+_n\xrightharpoonup[n\to\infty]{L^2(0,T;H)}\Psi \E \dot{u}.
\end{align*}
It remains to study the behaviour as $n\to \infty$ of the terms
\begin{equation*}
    \int_0^{t_n}(\dot{\tilde u}_n,\tilde w_n^+)_H\,\de s,\qquad\int_0^{t_n}(g^+_n,\tilde u_n^+-\tilde{ w}^+_n)_{H_N}\,\de s.
\end{equation*}
Thanks to formula~\eqref{dwn} we have
\begin{equation*}
    \int_0^{t_n}(\dot{\tilde u}_n,\tilde w_n^+)_H\,\de s=(\tilde u_n^+(t),\tilde w_n^+(t))_H-(u^1,\dot w(0))_H-\int_0^{t_n}(\tilde u_n^-,\dot{\tilde w}_n)_H\,\de s.
\end{equation*}
By arguing as before we hence deduce
\begin{equation}
 \int_0^{t_n}(\dot{\tilde u}_n,\tilde w_n^+)_H\,\de s\xrightarrow[n\to \infty]{}(\dot u(t),\dot w(t))_H-(u^1,\dot w(0))_H-\int_0^{t}(\dot u,\ddot w)_H\,\de s,
\end{equation}
thanks to~\eqref{newconv} and by these convergences
\begin{align*}
&\chi_{[0,t_n]}\dot{\tilde w}_n\xrightarrow[n\to\infty]{L^2(0,T;H)}\chi_{[0,t]}\ddot{w}, \quad \tilde{u}^-_n\xrightharpoonup[n\to\infty]{L^2(0,T;H)}\dot{u},\nonumber \\
&\|\tilde{w}_n^+(t)-\dot{w}(t)\|_H=\left\|\frac{w(j\tau_n)-w((j-1)\tau_n)}{\tau_n}-\dot{w}(t)\right\|_H=\left\|\dashint_{(j-1)\tau_n}^{j\tau_n}(\dot{w}(s)-\dot{w}(t))\,\de s\right\|_H\nonumber \\
&\hspace{2.63cm}\leq \dashint_{(j-1)\tau_n}^{j\tau_n}\|\dot{w}(s)-\dot{w}(t)\|_H\,\de s\xrightarrow[n \to \infty]{}0.
\end{align*}
Notice that in the last convergence we used the continuity of $w$ from $[0,T]$ in $H$. Similarly we have
\begin{equation*}
    \int_0^{t_n}(g_n^+,\tilde u_n^+-\tilde w_n^+)_{H_N}\,\de s=(g_n^+(t),u_n^+(t)- w_n^+(t))_{H_N}-(g(0),u^0- w(0))_{H_N}-\int_0^{t_n}(\dot g_n,u_n^--w_n^-)_{H_N}\,\de s
\end{equation*}
so that we get
\begin{equation}\label{se7}
\int_0^{t_n}(g_n^+,\tilde u_n^+-\tilde w_n^+)_{H_N}\,\de s\xrightarrow[n\to \infty]{}(g(t),u(t)- w(t))_{H_N}-(g(0),u^0- w(0))_{H_N}-\int_0^t(\dot g,u-w)_{H_N}\,\de s
\end{equation}
thanks to~\eqref{newconv}, the continuity of $s\mapsto g(s)$ in $H_N$, and the fact that
\begin{align*}
\chi_{[0,t_n]}\dot g_n\xrightarrow[n\to\infty]{L^2(0,T;H_N)}\chi_{[0,t]}\dot{g}, \qquad u_n^--w_n^-\xrightharpoonup[n\to\infty]{L^2(0,T;H_N)}u-w.
\end{align*}
By combining~\eqref{se1}--\eqref{se7}, we deduce
%\begin{align*}
%&\frac{1}{2}\| \dot{u}(t)\|^2_H+\frac{1}{2}(\mathbb{C}\E u(t),\E u(t))_H+\int_0^{t} (\mathbb{B}\Psi \E \dot{u},\Psi \E \dot{u})_H\,\de s \\
%&\hspace{0.8cm}\leq \mathcal{E}(0)+\int_0^{t}(f,\dot{ u}-\dot w)_H\,\de s+\int_0^{t}(\mathbb{C}\E u,\E \dot{ w})_H\,\de s+\int_0^{t}(\mathbb{B}\Psi \E \dot{u},\Psi \E \dot{ w})_H\,\de s -\int_0^{t}(\dot{u},\ddot{ w})_H\,\de s + (\dot{u}(t),\dot{w}(t))_H \\
%&\hspace{0.8cm}-\int_0^{t}(\dot{g}, u-w)_{H_N}\,\de s +(g(t), u(t)-w(t))_{H_N}-(u^1,\dot{w}(0))_H-(g(0), u^0-w(0))_{H_N},
%\end{align*}
%that is 
the energy--dissipation inequality~\eqref{eq:enin} for every $t\in(0,T]$. Finally, for $t=0$ the inequality trivially holds since $u(0)=u^0$ and $\dot u(0)=u^1$. 
\end{proof}

We now are in position to prove the validity of the initial conditions.

\begin{lemma}\label{lem:ic}
The weak solution $u\in\mathcal W$ to~\eqref{visco_elasto_system}--\eqref{boundary} of Lemma~\ref{lem:exis} satisfies
\begin{equation}\label{eq:incon}
    \lim_{t\to 0^+}u(t)= u^0\text{ in $V_T$},\quad\lim_{t\to 0^+}\dot u(t)=u^1\text{ in $H$}.
\end{equation}
In particular $u$ satisfies the initial conditions~\eqref{initials} in the sense of~\eqref{icd}.
\end{lemma}

\begin{proof}
By sending $t\rightarrow0^+$ into the energy--dissipation inequality ~\eqref{eq:enin} and using that $u\in C_w^0([0,T];V_T)$ and $\dot u\in C_w^0([0,T];H)$ we deduce
\begin{equation*}
    \mathcal E(0)\le\liminf_{t\to 0^+}\mathcal E(t)\le \limsup_{t\to 0^+}\mathcal E(t)\le \mathcal E(0),
\end{equation*}
since the right--hand side of~\eqref{eq:enin} is continuous in $t$, $u(0)=u^0$, and $\dot u(0)=u^1$.
Therefore there exists $\lim_{t\to 0^+}\mathcal E(t)=\mathcal E(0)$. By using the lower semicontinuity of $t\mapsto \norm{\dot u(t)}^2_H$ and $t\mapsto (\mathbb C\E u(t),\E u(t))_H$, we derive  
\begin{equation*}
\lim_{t\to 0^+}\norm{\dot u(t)}_H^2=\norm{u^1}_H^2,\quad \lim_{t\to 0^+}(\mathbb C\E u(t),\E u(t))_H=(\mathbb C\E u^0,\E u^0)_H.
\end{equation*}
Finally, since we have 
$$\dot u(t)\xrightharpoonup[t\to 0^+]{H}u^1,\quad  \E u(t)\xrightharpoonup[t\to 0^+]{H}\E u^0,$$ 
we deduce~\eqref{eq:incon}. In particular the functions $u\colon [0,T]\to V_T$ and $\dot u\colon [0,T]\to H$ are continuous at $t=0$, which implies~\eqref{icd}.
\end{proof}

We can finally prove Theorem~\ref{thm:mainresult}.

\begin{proof}[Proof of Theorem~\ref{thm:mainresult}]
It is enough to combine Lemmas~\ref{lem:exis} and~\ref{lem:ic}.
\end{proof}

\begin{remark}
We have proved Theorem~\ref{thm:mainresult} for the $d$-dimensional linear elastic case, namely when the displacement $u$ is a vector--valued function. The same result is true with identical proofs in the antiplane case, that is when the displacement $u$ is a scalar function and satisfies~\eqref{viscoelastic}.
\end{remark}

\section{Uniqueness}\label{visc_sec4}
In this section we investigate the uniqueness properties of system~\eqref{visco_elasto_system} with boundary and initial conditions~\eqref{boundary0}--\eqref{initials}. To this aim, we need to assume stronger regularity assumptions on the crack sets $\{\Gamma_t\}_{t\in[0,T]}$ and on the function $\Psi$. Moreover, we have to restrict our problem to the dimensional case $d=2$, since in our proof we need to construct a suitable family of diffeomorphisms which maps the time--dependent crack $\Gamma_t$ into a fixed set, and this can be explicitly done only for $d=2$ (see~\cite[Example 2.14]{DM-Luc}). 

We proceed in two steps; first, in Lemma~\ref{lem:uniq1} we prove a uniqueness result in every dimension $d$, but when the cracks are not increasing, that is $\Gamma_T=\Gamma_0$. Next, in Theorem~\ref{thm:uniq} we combine Lemma~\ref{lem:uniq1} with the finite speed of propagation theorem of~\cite{DMLT2} and the uniqueness result of~\cite{DMT2} to derive the uniqueness of a weak solution to~\eqref{visco_elasto_system}--\eqref{initials} in the case $d=2$.
 
Let us start with the following lemma, whose proof is similar to that one of~\cite[Proposition 2.10]{DMT2}.

\begin{lemma}\label{equiv}
	Let $u\in \mathcal{W}$ be a weak solution to~\eqref{visco_elasto_system}--\eqref{boundary} satisfying the initial condition $\dot u(0)=0$ in the following sense
	\begin{equation*}
	\lim_{h\to 0^+}\frac{1}{h}\int_0^h\norm{\dot u(t)}_H^2 =0.
	\end{equation*}
	Then $u$ satisfies
	\begin{equation*}
	\begin{aligned}
	&-\int_0^T(\dot u(t),\dot\varphi(t))_H \,\de t+\int_0^T(\mathbb C\E u(t),\E \varphi(t))_{H} \,\de t
	+\int_0^T(\mathbb B\Psi(t)\E \dot u(t),\Psi(t)\E \varphi(t))_{H} \,\de t\\
	&=\int_0^T(f(t),\varphi(t))_H \,\de t+\int_0^T(g(t),\varphi(t))_{H_N} \,\de t
	\end{aligned}
	\end{equation*}
	for every $\varphi\in\mathcal V^D$ such that $\varphi(T)=0$, where $\Psi \E \dot u$ is the function defined in~\eqref{Psiedotu}.
\end{lemma}

\begin{proof}
We fix $\varphi\in \mathcal V^D$ with $\varphi(T)=0$ and for every $\e>0$ we define the following function
\begin{equation*}
\varphi_{\epsilon}(t):=
\begin{cases}
\frac{t}{\epsilon}\varphi(t) &t\in [0,\epsilon],\\
\varphi(t)&t\in [\epsilon,T].
\end{cases}
\end{equation*}
We have that $\varphi_{\epsilon}\in \mathcal V^D$ and $\varphi_{\epsilon}(0)=\varphi_{\epsilon}(T)=0$, so we can use $\varphi_{\epsilon}$ as test function in~\eqref{weakform}. By proceeding as in~\cite[Proposition 2.10]{DMT2} we obtain
\begin{align*}
&\lim_{\e\to 0^+}\int_0^T(\dot u(t),\dot\varphi_{\epsilon}(t))_H \,\de t=\int_0^T(\dot u(t),\dot\varphi(t))_H \,\de t,\\
&\lim_{\e\to 0^+}\int_0^T(\mathbb{C} \E u(t),\E \varphi_{\epsilon}(t))_H \,\de t= \int_0^T(\mathbb{C} \E u(t),\E \varphi(t))_H \,\de t,\\
&\lim_{\e\to 0^+}\int_0^T(f(t),\varphi_{\epsilon}(t))_H \,\de t= \int_0^T(f(t),\varphi(t))_H \,\de t.
\end{align*}
It remains to consider the terms involving $\mathbb B$ and $g$. 
	We have
\begin{align*}	
&\int_0^T(\mathbb B\Psi(t)\E \dot u(t),\Psi(t)\E \varphi_{\epsilon}(t))_{H} \,\de t=\int_0^{\epsilon}(\mathbb B\Psi(t)\E \dot u(t),\frac{t}{\e}\Psi(t)\E \varphi(t))_{H} \,\de t+\int_{\epsilon}^{T}(\mathbb B\Psi(t)\E \dot u(t),\Psi(t)\E \varphi(t))_{H} \,\de t,\\
&\int_0^T(g(t),\varphi_{\epsilon}(t))_{H_N} \,\de t =\int_0^{\epsilon}(g(t),\frac{t}{\e}\varphi(t))_{H_N} \,\de t+\int_{\epsilon}^{T}(g(t),\varphi(t))_{H_N} \,\de t,
\end{align*}
hence by the dominated convergence theorem we get
\begin{align*}
&\int_{\epsilon}^{T}(\mathbb B\Psi(t)\E \dot u(t),\Psi(t)\E \varphi(t))_{H} \,\de t\xrightarrow[\epsilon\to 0^+]{}\int_{0}^{T}(\mathbb B\Psi(t)\E \dot u(t),\Psi(t)\E \varphi(t))_{H} \,\de t,\\ 
&\left |\int_0^{\epsilon}(\mathbb B\Psi(t)\E \dot u(t),\frac{t}{\e}\Psi(t)\E \varphi(t))_{H} \,\de t\right|\leq \norm{\mathbb B}_\infty\norm{\Psi}_\infty\int_0^{\epsilon}\norm{\Psi(t)\E \dot u(t)}_H\norm{\E \varphi(t)}_{H} \,\de t\xrightarrow[\e\to 0^+]{}0,\\
&\int_{\epsilon}^{T}(g(t),\varphi(t))_{H_N} \,\de t\xrightarrow[\epsilon\to 0^+]{}\int_{0}^{T}(g(t),\varphi(t))_{H_N} \,\de t,\\
&\left |\int_0^{\epsilon}(g(t),\frac{t}{\e}\varphi(t))_{H_N} \,\de t\right|\leq \int_0^{\epsilon}\norm{g(t)}_{H_N}\norm{\varphi(t)}_{H_N} \,\de t\xrightarrow[\e\to 0^+]{}0.
\end{align*}
	% \begin{align*}
	% &\dashint_0^{\epsilon}(\mathbb{C} \E u(t),t \E \varphi(t))_H\,\de t\xrightarrow[\epsilon\to 0^+]{}0,\quad \quad \quad \dashint_0^{\epsilon}(\mathbb B\Psi(t)\E \dot u(t),t\Psi(t)\E \varphi(t))_{H} \,\de t\xrightarrow[\epsilon\to 0^+]{}0\\
	% &\dashint_{T-\epsilon}^{T}(\mathbb C \E u(t),(T-t)\E \varphi(t))_{H} \,\de t\xrightarrow[\epsilon\to 0^+]{}0,\quad \dashint_{T-\epsilon}^T (\mathbb B\Psi(t)\E \dot u(t),(T-t)\Psi(t)\E \varphi(t))_{H} \,\de t\xrightarrow[\epsilon\to 0^+]{}0\\
	% &\int_{\epsilon}^{T-\epsilon}(\mathbb{C} \E u(t), \E \varphi(t))_H \,\de t\xrightarrow[\epsilon\to 0^+]{}\int_{0}^{T}(\mathbb{C} \E u(t), \E \varphi(t))_H \,\de t,\\
	% &\int_{\epsilon}^{T-\epsilon}(\mathbb B\Psi(t)\E \dot u(t),\Psi(t)\E \varphi(t))_{H} \,\de t\xrightarrow[\epsilon\to 0^+]{}\int_{0}^{T}(\mathbb B\Psi(t)\E \dot u(t),\Psi(t)\E \varphi(t))_{H} \,\de t.
	% \end{align*}
	% Finally, thanks to these convergences, we can claim that
	% \begin{equation}\label{lem2}
	% \begin{aligned}
	%     &\int_0^T(\mathbb{C} \E u(t),\E \varphi_{\epsilon}(t))_H \,\de t\xrightarrow[\epsilon\to 0^+]{} \int_0^T(\mathbb{C} \E u(t),\E \varphi(t))_H \,\de t,\\
	%     &\int_0^T(\mathbb B\Psi(t)\E \dot u(t),\Psi(t)\E \varphi_{\epsilon}(t))_{H} \,\de t\xrightarrow[\epsilon\to 0^+]{}\int_0^T(\mathbb B\Psi(t)\E \dot u(t),\Psi(t)\E \varphi(t))_{H} \,\de t,
	%   \end{aligned}
	% \end{equation}
	% so~\eqref{lem1} and~\eqref{lem2} imply the thesis.
By combining together all the previous convergences we get the thesis.
\end{proof}

We now state the uniqueness result in the case of a fixed domain, that is $\Gamma_T=\Gamma_0$. We follow the same ideas of~\cite{Ld}, and we need to assume
\begin{equation}\label{Psi2}
\Psi\in \Lip([0,T]\times \overline\Omega),\quad \nabla\dot\Psi\in L^\infty((0,T)\times \Omega;\R^d),
\end{equation}
while on $\Gamma_0$ we do not require any further hypotheses.

\begin{lemma}[Uniqueness in a fixed domain]\label{lem:uniq1}
Assume~\eqref{Psi2} and $\Gamma_T=\Gamma_0$. Then the viscoelastic dynamic system~\eqref{visco_elasto_system} with boundary and initial conditions~\eqref{boundary0}--\eqref{initials} $($the latter in the sense of~\eqref{icd}$)$ has a unique weak solution.
\end{lemma}

\begin{proof}
Let $u_1,u_2\in \mathcal{W}$ be two weak solutions to~\eqref{visco_elasto_system}--\eqref{boundary} with initial conditions~\eqref{initials}. The function $u:=u_1-u_2$ satisfies 
\begin{equation}\label{icd0}
\frac{1}{h}\int_0^h(\norm{u(t)}_{V_t}^2+\norm{\dot u(t)}_H^2) \,\de t\xrightarrow[h\to 0^+]{} 0,
\end{equation}
hence by Lemma~\ref{equiv} it solves
\begin{equation}\label{eq:uniq}
-\int_0^T(\dot u(t),\dot\varphi(t))_H \,\de t+\int_0^T(\mathbb C\E u(t),\E \varphi(t))_{H} \,\de t
+\int_0^T(\mathbb B\Psi(t)\E \dot u(t),\Psi(t)\E \varphi(t))_{H} \,\de t=0
\end{equation}
for every $\varphi\in\mathcal V^D$ such that $\varphi(T)=0$. We fix $s\in (0,T]$ and consider the function
\begin{equation*}
\varphi_{s}(t):=
\begin{cases}
-\int_t^s u(\tau)\de\tau \quad &t\in [0,s],\\
0\quad &t\in [s,T].
\end{cases}
\end{equation*}
Since $\varphi_s\in\mathcal V^D$ and $\varphi_s(T)=0$, we can use it as test function in~\eqref{eq:uniq} to obtain
\begin{equation*}
-\int_0^s(\dot u(t),u(t))_H \,\de t+\int_0^s(\mathbb C\E \dot{\varphi_s}(t),\E \varphi_s(t))_{H} \,\de t
+\int_0^s(\mathbb B\Psi(t)\E \dot u(t),\Psi(t)\E \varphi_s(t))_{H} \,\de t=0.
\end{equation*}
In particular we deduce 
\begin{equation*}
-\frac{1}{2}\int_0^s\frac{\de}{\de t}\|u(t)\|^2_H \,\de t+\frac{1}{2}\int_0^s\frac{\de}{\de t}(\mathbb C\E \varphi_s(t),\E \varphi_s(t))_{H} \,\de t
+\int_0^s(\mathbb B\Psi(t)\E \dot u(t),\Psi(t)\E \varphi_s(t))_{H} \,\de t=0,
\end{equation*}
which implies
\begin{equation}\label{uni0}
\frac{1}{2}\|u(s)\|^2_H+\frac{1}{2}(\mathbb C\E \varphi_s(0),\E \varphi_s(0))_{H}=\int_0^s(\mathbb B\Psi (t)\E \dot u(t) ,\Psi(t) \E \varphi_s(t) )_{H} \,\de t,
\end{equation}
since $u(0)=0=\varphi_s(s)$. From the distributional point of view the following equality holds
\begin{equation}\label{distr}
\frac{\de}{\de t}(\Psi \E u)=\dot{\Psi}\E u+\Psi \E \dot{u}\in L^2(0,T;H),
\end{equation}
indeed, for all $v\in C^{\infty}_c(0,T;H)$ we have
\begin{align*}
&\int_0^T(\frac{\de}{\de t}(\Psi(t) \E u(t)),v(t))_H\de t=-\int_0^T\left(\Psi(t) \E u(t),\dot{v}(t)\right)_H\,\de t\\
&=-\int_0^T\left(\E (\Psi(t) u(t))-\nabla \Psi(t)\odot u(t),\dot{v}(t)\right)_H\,\de t\\
&=\int_0^T (\E (\dot{\Psi}(t)u(t))+\E (\Psi(t) \dot{u}(t)),v(t))_H\,\de t-\int_0^T (\nabla \dot{\Psi}(t)\odot u(t)+\nabla \Psi(t)\odot \dot{u}(t),v(t))_H\,\de t\\
&=\int_0^T(\dot{\Psi}(t) \E u(t),v(t))_H\,\de t+\int_0^T(\Psi(t) \E \dot{u}(t),v(t))_H\,\de t.
\end{align*}
In particular $\Psi \E u\in H^1(0,T;H)\subset C^0([0,T],H)$, so that by~\eqref{icd0}
\begin{equation*}
\|\Psi(0)\E u(0)\|^2_H=\lim_{h\to 0}\frac{1}{h}\int_0^h  \|\Psi(t)\E u(t)\|^2_H \de t\leq C\lim_{h\to 0}\frac{1}{h}\int_0^h  \|u(t)\|^2_{V_t} \de t=0
\end{equation*}	which yields $\Psi(0)\E u(0)=0$. Thanks to~\eqref{distr} and to property $\Psi u\in H^1(0,T;H)$, we deduce
\begin{equation*}
	\begin{aligned}
	\frac{\de}{\de t}\left(\mathbb{B}\Psi \E u ,\Psi \E \varphi_s \right)_H &=(\mathbb{B}\dot{\Psi} \E u ,\Psi \E \varphi_s )_H+(\mathbb{B}\Psi \E \dot{ u} ,\Psi \E \varphi_s )_H+(\mathbb{B}\Psi \E u ,\dot{\Psi} \E \varphi_s )_H+(\mathbb{B}\Psi \E u ,\Psi \E \dot{\varphi}_s )_H\\
	&=2(\mathbb{B}\Psi \E u ,\dot\Psi \E \varphi_s )_H+(\mathbb{B}\Psi \E \dot{ u} ,\Psi \E \varphi_s )_H+(\mathbb{B}\Psi \E u ,\Psi \E \dot{\varphi}_s )_H,
	\end{aligned}
	\end{equation*}
	and by integrating on $[0,s]$ we get
	\begin{align*}
	&\int_0^s  (\mathbb{B}\Psi(t)\E \dot{ u}(t),\Psi(t)\E \varphi_s(t))_H\,\de t\\
	&=\int_0^s \left[\frac{\de}{\de t}(\mathbb{B}\Psi(t)\E u(t),\Psi(t)\E \varphi_s(t))_H-2 (\mathbb{B}\Psi(t)\E u(t),\dot\Psi(t)\E \varphi_s(t))_H- (\mathbb{B}\Psi(t)\E \dot{\varphi}_s(t),\Psi(t)\E \dot{\varphi}_s(t))_H\right]\de t\\
	&\leq (\mathbb{B}\Psi(s)\E u(s),\Psi(s)\E \varphi_s(s))_H-(\mathbb{B}\Psi(0)\E u(0),\Psi(0)\E \varphi_s(0))_H\\
	&\quad+\int_0^s\left[2(\mathbb{B}\Psi(t)\E u(t),\Psi(t)\E u(t))^{\frac{1}{2}}_H(\mathbb{B}\dot{\Psi}(t)\E \varphi_s(t),\dot{\Psi}(t)\E \varphi_s(t))^{\frac{1}{2}}_H- (\mathbb{B}\Psi(t)\E \dot{\varphi}_s(t),\Psi(t)\E \dot{\varphi}_s(t))_H\right]\de t\\
	&\leq \int_0^s\left[ (\mathbb{B}\Psi(t)\E u(t),\Psi(t)\E u(t))_H+(\mathbb{B}\dot{\Psi}(t)\E \varphi_s(t),\dot{\Psi}(t)\E \varphi_s(t))_H-(\mathbb{B}\Psi(t)\E \dot{\varphi}_s(t),\Psi(t)\E \dot{\varphi}_s(t))_H\right]\de t\\
	&\leq \norm{\mathbb B}_\infty\norm{\dot\Psi}_\infty^2\int_0^s\|\E \varphi_s(t)\|_H^2\de t,
	\end{align*}
	since $\E \varphi_s(s)=0=\Psi(0)\E u(0)$ and $\E \dot \varphi_s=\E u$ in $(0,s)$. By combining the previous inequality with~\eqref{uni0} and using the coercivity of the tensor $\mathbb{C}$, we derive
\begin{align*}
\frac{\lambda_1}{2}\|\E \varphi_s(0)\|^2_H+\frac{1}{2}\|u(s)\|_H^2&\le \frac{1}{2}(\mathbb C\E \varphi_s(0),\E \varphi_s(0))_H+\frac{1}{2}\|u(s)\|_H^2\le \norm{\mathbb B}_\infty\norm{\dot\Psi}_\infty^2\int_0^s\|\E \varphi_s(t)\|_H^2\de t.
\end{align*}

Let us set $\xi(t)\coloneqq \int_0^t u(\tau)\de \tau$, then
\begin{align*}
&\|\E \varphi_s(0)\|_H^2=\|\E \xi(s)\|^2_H,\quad\|\E \varphi_s(t)\|^2_H=\|\E \xi(t)-\E \xi(s)\|_H^2\le 2\|\E \xi(t)\|^2_H+2\|\E \xi(s)\|_H^2,
\end{align*}
from which we deduce
\begin{equation}\label{ul}
\frac{\lambda_1}{2}\|\E \xi(s)\|^2_H+\frac{1}{2}\norm{u(s)}_H^2\le C\int_0^s\|\E \xi(t)\|^2_H\de t+Cs\|\E \xi(s)\|^2_H,
\end{equation}
where $C:=2\norm{\mathbb B}_\infty\norm{\dot\Psi}_\infty^2$. Therefore, if we set $s_0\coloneqq \frac{\lambda_1}{4C}$, for all $s\le s_0$ we obtain 
\begin{equation*}
\frac{\lambda_1}{4}\|\E \xi(s)\|^2_H\le \left(\frac{\lambda_1}{2}- Cs\right)\|\E \xi(s)\|^2_H\le C\int_0^s\|\E \xi(t)\|^2_H\de t.
\end{equation*}
By Gronwall's lemma the last inequality implies $\E \xi(s)=0$ for all $s\le s_0$. Hence, thanks to~\eqref{ul} we get $\|u(s)\|^2_H\le0$ for all $s\le s_0$, which yields $u(s)=0$ for all $s\le s_0$. Since $s_0$ depends only on $\mathbb C$, $\mathbb B$, and $\Psi$, we can repeat this argument starting from $s_0$, and with a finite number of steps we obtain $u\equiv 0$ on $[0,T]$.
\end{proof}

In order to prove our uniqueness result in the case of a moving crack we need two auxiliary results, which are~\cite[Theorem 6.1]{DMLT} and ~\cite[Theorem 4.3]{DMT2}. For the sake of the readers, we rewrite below the statements without proof.

The first one (\cite[Theorem 6.1]{DMLT}) is a generalization of the well--known result of finite speed of propagation for the wave equation. Given an open bounded set $U\subset\R^d$, we define by $\partial_LU$ the Lipschitz part of the boundary $\partial U$, which is the collection of points $x\in \partial U$ for which there exist an orthogonal coordinate system $y_1,\dots,y_d$, a neighborhood $V$ of $x$ of the form $A\times I$, with $A$ open in $\R^{d-1}$ and $I$ open interval in $\R$, and a Lipschitz function $g\colon A\to I$, such that $V\cap U:=\{(y_1,\dots,y_d)\in V:y_d<g(y_1,\dots,y_{d-1})\}$. Moreover, given a Borel set $S\subseteq\partial_LU$, we define
\begin{equation*}
    H_S(U;\R^d):=\{u\in H^1(U;\R^d):u=0\text{ on }S\}.
\end{equation*}
Notice that $H_S(U;\R^d)$ is a Hilbert space, and we denote its dual by $H^{-1}_S(U;\R^d).$

\begin{theorem}[Finite speed of propagation]\label{finprop}
Let $U\subset\R^d$ be an open bounded set and let $\partial_LU$ be the Lipschitz part of $\partial U$. Let $S_0$ and $S_1$ be two Borel sets with $S_0\subseteq S_1\subseteq\partial_LU$, and let $\mathbb C\colon U\to\mathscr L(\R^{d\times d}_{sym};\mathbb R^{d\times d}_{sym})$ be a fourth-order tensor satisfying~\eqref{CB1}--\eqref{CB2}. Let
\begin{equation*}
    u\in L^2(0,T;H^1_{S_0}(U;\R^d))\cap H^1(0,T;L^2(U;\R^d))\cap H^2(0,T;H^{-1}_{S_1}(U;\R^d))
\end{equation*}
be a solution to
\begin{equation*}
    \spr{\ddot u(t)}{\psi}_{H^{-1}_{S_1}(U;\R^d)}+(\mathbb C\E u(t),\E \psi)_{L^2(U;\R^{d\times d}_{sym})}=0\quad\text{for every }\psi\in H^1_{S_1}(U;\R^d),
\end{equation*}
with initial conditions $u(0)=0$ and $\dot u(0)=0$ in the sense of $L^2(U;\R^d)$ and $H^{-1}_{S_1}(U;\R^d)$, respectively. Then 
\begin{equation*}
    u(t)=0\quad\text{a.e. in }U_t:=\{x\in U:\dist(x,S_1\setminus S_0)>t\sqrt{\norm{\mathbb C}_\infty}\}
\end{equation*}
for every $t\in[0,T]$.
\end{theorem}

\begin{proof}
See~\cite[Theorem 6.1]{DMLT}.
\end{proof}

The second one (\cite[Theorem 4.3]{DMT2}) is a uniqueness result for the weak solutions of the wave equation in a moving domain. Let $\hat H$ be a separable Hilbert space, and let $\{\hat V_t\}_{t\in[0,T]}$ be a family of separable Hilbert spaces with the following properties:
\begin{itemize}
    \item[$(i)$] for every $t\in[0,T]$ the space $\hat V_t$ is contained and dense in $\hat H$ with continuous embedding;
    \item[$(ii)$] for every $s,t\in[0,T]$, with $s<t$, $\hat V_s\subset \hat V_t$ and the Hilbert space structure on $\hat V_s$ is the one induced by $\hat V_t$.
\end{itemize}
Let $a\colon \hat V_T\times \hat V_T\to\R$ be a bilinear symmetric form satisfying the following conditions:
\begin{itemize}
    \item[$(iii)$] there exists $M_0$ such that
    \begin{equation*}
        |a(u,v)|\le M_0\norm{u}_{\hat V_T}\norm{v}_{\hat V_T}\quad\text{for every }u,v\in\hat V_T;
    \end{equation*}
    \item[$(iv)$] there exist $\lambda_0>0$ and $\nu_0\in\R$ such that
    \begin{equation*}
        a(u,u)\ge \lambda_0\norm{u}_{\hat V_T}^2-\nu_0\norm{u}_{\hat H}^2\quad\text{for every }u\in \hat V_T.
    \end{equation*}
\end{itemize}
Assume that 
\begin{itemize}
    \item[(U1)] for every $t\in[0,T]$ there exists a continuous and linear bijective operator $Q_t\colon \hat V_t\to \hat V_0$, with continuous inverse $R_t\colon\hat V_0\to\hat V_t$;
    \item[(U2)] $Q_0$ and $R_0$ are the identity maps on $\hat V_0$;
    \item[(U3)] there exists a constant $M_1$ independent of $t$ such that 
    \begin{align*}
        &\norm{Q_tu}_{\hat H}\le M_1 \norm {u}_{\hat H}\quad\text{for every }u\in \hat V_t,& &        \norm{R_tu}_{\hat H}\le M_1 \norm {u}_{\hat H}\quad\text{for every }u\in \hat V_0,\\
        &\norm{Q_tu}_{\hat V_0}\le M_1 \norm {u}_{\hat V_t}\quad\text{for every }u\in \hat V_t,& &        \norm{R_tu}_{\hat V_t}\le M_1 \norm {u}_{\hat V_0}\quad\text{for every }u\in \hat V_0.
    \end{align*}
\end{itemize}
Since $\hat V_t$ is dense in $\hat H$, (U3) implies that $R_t$ and $Q_t$ can be extended to continuous linear operators from $\hat H$ into itself, still denoted by $Q_t$ and $R_t$. We also require
\begin{itemize}
    \item[(U4)] for every $v\in \hat V_0$ the function $t\mapsto R_tv$ from $[0,T]$ into $\hat H$ has a derivative, denoted by $\dot R_t v$;
    \item[(U5)] there exists $\eta\in(0,1)$ such that 
    \begin{equation*}
        \norm{\dot R_t Q_tv}^2_{\hat H}\le \lambda_0(1-\eta)\norm{v}_{\hat V_t}^2\quad\text{for every }v\in \hat V_t;
    \end{equation*}
    \item[(U6)] there exists a constant $M_2$ such that 
    \begin{equation*}
        \norm{Q_tv-Q_sv}_{\hat H}\le M_2\norm{v}_{\hat V_s}(t-s)\quad\text{for every $0\le s<t\le T$ and every $v\in \hat V_s$};
    \end{equation*}
    \item[(U7)] for very $t\in[0,T)$ and for every $v\in \hat V_t$ there exists an element of $\hat H$, denoted by $\dot Q_tv$, such that 
    \begin{equation*}
        \lim_{h\to 0^+}\frac{Q_{t+h}v-Q_tv}{h}=\dot Q_t v\text{ in }\hat H.
    \end{equation*}
\end{itemize}
For every $t\in[0,T]$, define
\begin{align*}
&\alpha(t)\colon \hat V_0\times \hat V_0\to\R\quad\text{as }\alpha(t)(u,v):=a(R_tu,R_tv)\text{ for }u,v\in\hat V_0,\\
&\beta(t)\colon \hat V_0\times \hat V_0\to\R\quad\text{as }\beta(t)(u,v):=(\dot R_tu,\dot R_tv)\text{ for }u,v\in\hat V_0,\\
&\gamma(t)\colon \hat V_0\times \hat H\to\R\quad\text{as }\gamma(t)(u,v):=(\dot R_tu,R_tv)\text{ for }u\in \hat V_0\text{ and }v\in\hat H,\\
&\delta(t)\colon \hat H\times \hat H\to\R\quad\text{as }\delta(t)(u,v):=(R_tu,R_tv)-(u,v)\text{ for }u,v\in\hat H.
\end{align*}
We assume that there exists a constant $M_3$ such that
\begin{itemize}
    \item[(U8)] the maps $t\mapsto \alpha(t)(u,v)$, $t\mapsto \beta(t)(u,v)$, $t\mapsto \gamma(t)(u,v)$, and $t\mapsto \delta(t)(u,v)$ are Lipschitz continuous and for a.e. $t\in(0,T)$ their derivatives satisfy
    \begin{align*}
    &|\dot\alpha(t)(u,v)|\le M_3\norm{u}_{\hat V_0}\norm{v}_{\hat V_0}\quad\text{for }u,v\in\hat V_0,\\
    &|\dot\beta(t)(u,v)|\le M_3\norm{u}_{\hat V_0}\norm{v}_{\hat V_0}\quad\text{for }u,v\in\hat V_0,\\
    &|\dot\gamma(t)(u,v)|\le M_3\norm{u}_{\hat V_0}\norm{v}_{\hat H}\quad\text{for }u\in\hat V_0\text{ and }v\in\hat H,\\
    &|\dot\delta(t)(u,v)|\le M_3\norm{u}_{\hat H}\norm{v}_{\hat H}\quad\text{for }u,v\in\hat H.
    \end{align*}
\end{itemize}

\begin{theorem}[Uniqueness for the wave equation]\label{uniqwave}
Assume that $\hat H$, $\{\hat V_t\}_{t\in[0,T]}$, and $a$ satisfy $(i)$--$(iv)$ and that (U1)--(U8) hold. Given $u^0\in \hat V_0$, $u^1\in \hat H$, and $f\in L^2(0,T;\hat H)$, there exists a unique solution
\begin{equation*}
u\in \hat{\mathcal V}:=\{\varphi\in L^2(0,T;\hat V_T):\dot u\in L^2(0,T;\hat H),\,u(t)\in \hat V_t\text{ for a.e. }t\in(0,T)\}
\end{equation*}
to the wave equation
\begin{equation*}
    -\int_0^T(\dot u(t),\dot \varphi(t))_{\hat H}\,\de t+\int_0^Ta(u(t),\varphi(t))\,\de t=\int_0^T(f(t),\varphi(t))_{\hat H}\,\de t\quad\text{for every $\varphi\in\hat{\mathcal V}$},
\end{equation*}
satisfying the initial conditions $u(0)=u^0$ and $\dot u(0)=u^1$ in the sense that 
\begin{equation*}
\lim_{h\to 0^+}\frac{1}{h}\int_0^h\left(\|u(t)-u^0\|^2_{\hat V_t}+\|\dot{u}(t)-u^1\|^2_{\hat H}\right) \,\de t=0.
\end{equation*}
\end{theorem}

\begin{proof}
See~\cite[Theorem 4.3]{DMT2}.
\end{proof}

We now are in position to prove the uniqueness theorem in the case of a moving domain. We consider the dimensional case $d=2$, and we require the following assumptions:
\begin{itemize}
	\item [(H1)] there exists a $C^{2,1}$ simple curve $\Gamma\subset\overline\Omega\subset\R^2$, parametrized by arc--length $\gamma\colon [0,\ell]\to\overline\Omega$, such that $\Gamma\cap\partial\Omega=\gamma(0)\cup\gamma(\ell)$ and $\Omega\setminus\Gamma$ is the union of two disjoint open sets with Lipschitz boundary;
	\item [(H2)] there exists a non decreasing function $s\colon [0,T]\to (0,\ell)$ of class $C^{1,1}$ such that $\Gamma_t=\gamma([0,s(t)])$;
	\item [(H3)] $|\dot s(t)|^2< \frac{\lambda_1}{C_K}$, where $\lambda_1$ is the ellipticity constant of $\mathbb C$ and $C_K$ is the constant that appears in Korn's inequality in \eqref{eq:hyp_korn}.
\end{itemize}
Notice that hypotheses (H1) and (H2) imply (E1)--(E3). We also assume that $\Psi$ satisfies~\eqref{Psi2} and there exists a constant $\e>0$ such that for every $t\in[0,T]$
\begin{equation}\label{Psi3}
\Psi(t,x)=0\quad\text{for every  $x\in\{y\in\overline\Omega:|y-\gamma(s(t))|<\e\}$}.
\end{equation}

\begin{theorem}\label{thm:uniq}
Assume $d=2$ and (H1)--(H3),~\eqref{Psi2}, and~\eqref{Psi3}. Then the system~\eqref{visco_elasto_system} with boundary conditions~\eqref{boundary0}--\eqref{boundary} has a unique weak solution $u\in\mathcal W$ which satisfies $u(0)=u^0$ and $\dot u(0)=u^1$ in the sense of~\eqref{icd}.
\end{theorem}

\begin{proof}
	As before let $u_1,u_2\in\mathcal W$ be two weak solutions to~\eqref{visco_elasto_system}--\eqref{boundary} with initial conditions~\eqref{initials}. Then $u:=u_1-u_2$ satisfies~\eqref{icd0} and~\eqref{eq:uniq} for every $\varphi\in\mathcal V^D$ such that $\varphi(T)=0$. Let us define
	$$t_0:=\sup\{t\in[0,T]:u(s)=0\text{ for every }s\in[0,t] \},$$
	and assume by contradiction that $t_0<T$. Consider first the case in which $t_0>0$. By (H1), (H2),~\eqref{Psi2}, and~\eqref{Psi3} we can find two open sets $A_1$ and $A_2$, with $A_1\subset\subset A_2\subset\subset\Omega$, and a number $\delta>0$  such that for every $t\in[t_0-\delta,t_0+\delta]$ we have $\gamma(s(t))\in A_1$, $\Psi(t,x)=0$ for every $x\in\overline A_2$, and $(A_2\setminus A_1)\setminus \Gamma $ is the union of two disjoint open sets with Lipschitz boundary. Let us define
	\begin{equation*}
	\hat V^1:=\{u\in H^1((A_2\setminus A_1)\setminus\Gamma_{t_0-\delta};\R^2): u=0\text{ on }\partial A_1\cup\partial A_2\},\quad \hat H^1:= L^2(A_2\setminus A_1;\R^2).
	\end{equation*}
	Since every function in $\hat V^1$ can be extended to a function in $V_{t_0-\delta}^D$, by classical results for linear hyperbolic equations (se, e.g.,~\cite{DL}), we deduce 
	%$u$ satisfies 
	% \begin{equation}\label{eq:weak}
	%     -\int_{t_0}^{t_0+\delta}\int_{A_2\setminus A_1}\dot u(t,x)\cdot \dot \varphi(t,x)\,\de x\,\de t+\int_{t_0}^{t_0+\delta}\int_{A_2\setminus A_1}\mathbb C(x)\E u(t,x)\cdot \E \varphi(t,x)\,\de x\,\de t=0
	% \end{equation}
	% for every $\varphi\in L^2(t_0,t_0+\delta;\hat V^1)\cap H^1(t_0,t_0+\delta; \hat H^1)$ such that $\varphi(t_0)=\varphi(t_0+\delta)=0$. Therefore we can use~\cite[Theorem 2.16]{DMT2} to get that 
	$\ddot u\in L^2(t_0-\delta,t_0+\delta;(\hat V^1)')$ and that $u$ satisfies for a.e. $t\in(t_0-\delta,t_0+\delta)$ 
	\begin{equation*}
	\spr{\ddot u(t)}{\phi}_{(\hat V^1)'}+(\mathbb C\E u(t),\E \phi)_{ \hat H^1}=0\quad\text{for every }\phi\in \hat V^1.
	\end{equation*}
Moreover, we have $u(t_0)=0$ as element of $ \hat H^1$ and $\dot u(t_0)=0$ as element of $(\hat V^1)'$, since $u(t)\equiv 0$ in $[t_0-\delta,t_0)$, $u\in C^0([t_0-\delta,t_0];\hat H^1)$, and $\dot u\in C^0([t_0-\delta,t_0];(\hat V^1)')$. We are now in position to apply the result of finite speed of propagation of Theorem~\ref{finprop}. This theorem ensures the existence of a third open set $A_3$, with $A_1\subset\subset A_3\subset\subset A_2$, such that, up to choose a smaller $\delta$, we have $u(t)=0$ on $\partial A_3$ for every $t\in[t_0,t_0+\delta]$, and both $(\Omega\setminus A_3)\setminus\Gamma$ and $A_3\setminus \Gamma$ are union of two disjoint open sets with Lipschitz boundary.
	
In $\Omega\setminus A_3$ the function $u$ solves
	\begin{align*}
	-\int_{t_0-\delta}^{t_0+\delta}\int_{\Omega\setminus A_3}\dot u(t,x)\cdot \dot \varphi(t,x)\,\de x\,\de t&+\int_{t_0-\delta}^{t_0+\delta}\int_{\Omega\setminus A_3}\mathbb C(x)\E u(t,x)\cdot \E \varphi(t,x)\,\de x\,\de t\\
	&+\int_{t_0-\delta}^{t_0+\delta}\int_{\Omega\setminus A_3}\mathbb B(x)\Psi(t,x)\E \dot u(t,x)\cdot \Psi(t,x)\E \varphi(t,x)\,\de x\,\de t=0
	\end{align*}
	for every $\varphi\in L^2(t_0-\delta,t_0+\delta;\hat V^2)\cap H^1(t_0-\delta,t_0+\delta;\hat H^2)$ such that $\varphi(t_0-\delta)=\varphi(t_0+\delta)=0$, where 
	\begin{equation*}
	\hat V^2:=\{u\in H^1((\Omega\setminus A_3)\setminus\Gamma_{t_0-\delta};\R^2): u=0\text{ on }\partial_D\Omega\cup\partial A_3\},\quad \hat H^2:=L^2(\Omega\setminus A_3;\R^2).
	\end{equation*}
	Since $u(t)=0$ on $\partial_D\Omega\cup\partial A_3$ for every $t\in[t_0-\delta,t_0+\delta]$ and $u(t_0-\delta)=\dot u(t_0-\delta)=0$ in the sense of~\eqref{icd} (recall that $u\equiv 0$ in $[t_0-\delta,t_0)$), we can apply Lemma~\ref{lem:uniq1} to deduce $u(t)=0$ in $\Omega\setminus A_3$ for every $t\in[t_0-\delta,t_0+\delta]$.
	
	On the other hand in $A_3$, by setting
	\begin{equation*}
	\hat V^3_t:=\{u\in H^1(A_3\setminus\Gamma_t;\R^2): u=0\text{ on }\partial A_3\},\quad \hat H^3:= L^2(A_3;\R^2),
	\end{equation*}
	we get that the function $u$ solves
	\begin{align*}
	-\int_{t_0-\delta}^{t_0+\delta}\int_{A_3}\dot u(t,x)\cdot \dot \varphi(t,x)\,\de x\,\de t&+\int_{t_0-\delta}^{t_0+\delta}\int_{ A_3}\mathbb C(x)\E u(t,x)\cdot \E \varphi(t,x)\,\de x\,\de t=0
	\end{align*}
	for every $\varphi\in L^2(t_0-\delta,t_0+\delta;\hat V^3_{t_0+\delta})\cap H^1(t_0-\delta,t_0+\delta;\hat H^3)$ such that $\varphi(t)\in \hat V^3_t$ for a.e. $t\in(t_0-\delta,t_0+\delta)$ and $\varphi(t_0-\delta)=\varphi(t_0+\delta)=0$. Here we would like to apply the uniqueness result of Theorem~\ref{uniqwave} for the spaces $\{\hat V_t^3\}_{t\in[t_0-\delta,t_0+\delta]}$ and $\hat H^3$, endowed with the usual norms, and for the bilinear form
\begin{equation*}
a(u,v):=\int_{A_3}\mathbb C(x)\E u(x)\cdot \E v(x)\de x\quad\text{for every }u,v\in \hat V^3_{t_0+\delta}.
\end{equation*}
As show in~\cite[Example 2.14]{DM-Luc} we can construct two maps $\Phi,\Lambda\in C^{1,1}([t_0-\delta,t_0+\delta]\times \overline A_3;\R^2)$ such that for every $t\in[0,T]$ the function $\Phi(t,\cdot)\colon \overline A_3\to \overline A_3$ is a diffeomorfism of $A_3$ in itself with inverse $\Lambda(t,\cdot)\colon \overline A_3\to\overline A_3$. Moreover, $\Phi(0,y)=y$ for every $y\in \overline A_3$, $\Phi(t,\Gamma\cap \overline A_3)=\Gamma\cap\overline A_3$ and $\Phi(t,\Gamma_{t_0-\delta}\cap \overline A_3)=\Gamma_t\cap \overline A_3$ for every $t\in[t_0-\delta,t_0+\delta]$. For every $t\in[t_0-\delta,t_0+\delta]$, the maps $(Q_tu)(y):=u(\Phi(t,y))$, $u\in \hat V_t^3$ and $y\in A_3$,
and $(R_tv)(x):=v(\Lambda(t,x))$, $v\in \hat  V_{t_0-\delta}^3$ and $x\in A_3$, provide a family of linear and continuous operators which satisfy the assumptions (U1)--(U8) of Theorem~\ref{uniqwave} (see~\cite[Example 4.2]{DMT2}). The only condition to check is (U5). The bilinear form $a$ satisfies the following ellipticity condition
\begin{equation}\label{ella}
a(u,u)\ge \lambda_1\norm{\E u}_{L^2(A_3;\R^{2\times 2}_{sym})}^2\ge \frac{\lambda_1}{\hat C_k}\norm{u}_{\hat V_{t_0+\delta}^3}^2-\lambda_1\norm{u}_{\hat H^3}^2\quad\text{for every }u\in \hat V_{t_0+\delta}^3,
\end{equation}
where $\hat C_K$ is the constant in Korn's inequality in $\hat V_{t_0+\delta}^3$, namely
$$\norm{\nabla u}_{L^2(A_3;\R^{2\times 2})}^2\leq \hat C_K(\norm{u}_{L^2(A_3;\R^2)}^2+\norm{\E u}_{L^2(A_3;\R^{2\times 2}_{sym})}^2)\quad\text{for every }u\in \hat V_{t_0+\delta}^3.$$
Notice that for $t\in[t_0-\delta,t_0+\delta]$
\begin{equation*}
    (\dot R_tv)(x)=\nabla v(\Lambda(t,x))\dot\Lambda(t,x)\quad\text{for a.e. }x\in A_3,
\end{equation*}
from which we obtain 
\begin{equation*}
    \norm{\dot R_tQ_tu}_{\hat H^3}^2\le \int_{A_3}|\nabla u(x)|^2|\dot\Phi(t,\Lambda(t,x))|^2\,\de x.
\end{equation*}
Hence, have to show the property
\begin{equation*}
|\dot \Phi(t,y)|^2<\frac{\lambda_1}{\hat C_K}\quad\text{for every $t\in[t_0-\delta,t_0+\delta]$ and $y\in\overline A_3$}.
\end{equation*}
This is ensured by (H3). Indeed, as explained in~\cite[Example 3.1]{DM-Luc}, we can construct the maps $\Phi$ and $\Lambda$ in such a way that 
$$|\dot \Phi(t,y)|^2<\frac{\lambda_1}{C_K},$$
since $|\dot s(t)|^2<\frac{\lambda_1}{C_K}$. Moreover, every function in $\hat{V}^3_{t_0+\delta}$ can be extended to a function in $H^1(\Omega\setminus\Gamma;\R^d)$. Hence, for Korn's inequality in $\hat V_{t_0+\delta}^3$, we can use the same constant $C_K$ of $H^1(\Omega\setminus\Gamma;\R^d)$. This allows us to apply Theorem~\ref{uniqwave}, which implies $u(t)=0$ in $A_3$ for every $t\in[t_0,t_0+\delta]$. In the case $t_0=0$, it is enough to argue as before in $[0,\delta]$, by exploiting~\eqref{icd0}. Therefore $u(t)=0$ in $\Omega$ for every $t\in[t_0,t_0+\delta]$, which contradicts the maximality of $t_0$. Hence $t_0=T$, that yields $u(t)=0$ in $\Omega$ for every $t\in[0,T]$.
\end{proof}

\begin{remark}
Also Theorem~\ref{thm:uniq} is true in the antiplane case, with essentially the same proof. Notice that, when the displacement is scalar, we do not need to use Korn's inequality in~\eqref{ella} to get the coercivity in $\hat V_{t_0+\delta}^3$ of the bilinear form $a$ defined before. Therefore, in this case in (H3) it is enough to assume $|\dot s(t)|^2<\lambda_1$.
\end{remark}

\section{A Moving Crack Satisfying Griffith's Dynamic Energy--Dissipation Balance}\label{visc_sec5}

We conclude this paper with an example of a moving crack $\{\Gamma_t\}_{t\in[0,T]}$ and weak solution to~\eqref{visco_elasto_system}--\eqref{initials} which satisfy the energy--dissipation balance of Griffith's dynamic criterion, as happens in~\cite{DMLT} for the purely elastic case. In dimension $d=2$ we consider an antiplane evolution, which means that the displacement $u$ is scalar, and we take $\Omega:=\{x\in\R^2:|x|<R\}$, with $R>0$. We fix a constant $0<c<1$ such that $cT<R$, and we set 
\begin{equation*}
    \Gamma_t:=\{(\sigma,0)\in \overline\Omega\,:\sigma\leq ct\}.
\end{equation*}
Let us define the following function 
\begin{equation*}
    S(x_1,x_2):=Im(\sqrt{x_1+ix_2})=\frac{1}{\sqrt{2}}\frac{x_2}{\sqrt{|x|+x_1}}\quad x\in \R^2\setminus\{(\sigma,0):\sigma \le 0\},
\end{equation*}
where $Im$ denotes the imaginary part of a complex number. Notice that $S\in H^1(\Omega\setminus\Gamma_0)\setminus H^2(\Omega\setminus\Gamma_0)$, and it is a weak solution to 
\begin{equation*}
    \begin{cases}
    \Delta S=0 &\text{in $\Omega\setminus \Gamma_0$},\\
    \nabla S\cdot \nu=\partial_2S=0 &\text{on $\Gamma_0$}.
    \end{cases}
\end{equation*}
Let us consider the function 
\begin{equation*}
    u(t,x):=\frac{2}{\sqrt\pi}S\left(\frac{x_1-ct}{\sqrt{1-c^2}},x_2\right)\quad t\in[0,T],\,x\in\Omega\setminus\Gamma_t
\end{equation*}
and let $w(t)$ be its restriction to $\partial\Omega$. Since $u(t)$ has a singularity only at the crack tip $(ct,0)$, the function $w(t)$ can be seen as the trace on $\partial \Omega$ of a function belonging to $ H^2(0,T;L^2(\Omega))\cap H^1(0,T;H^1(\Omega\setminus\Gamma_0))$, still denoted by $w(t)$. It is easy to see that $u$ solves the wave equation
\begin{equation*}
    \ddot u(t)-\Delta u(t)=0\quad \text{in }\Omega\setminus\Gamma_t,\,t\in(0,T),
\end{equation*}
with boundary conditions
\begin{align*}
&u(t)=w(t)& &\hspace{-2,9cm}\text{on $\partial\Omega$, $t\in(0,T)$}, & &\\
&\frac{\partial u}{\partial \nu}(t)=\nabla u(t)\cdot \nu=0& &\hspace{-2,9cm}\text{on $\Gamma_t$, \,\, $t\in(0,T)$}, & & 
\end{align*}
and initial data
\begin{align*}
    &u^0(x_1,x_2):=\frac{2}{\sqrt\pi}S\left(\frac{x_1}{\sqrt{1-c^2}},x_2\right)\in H^1(\Omega\setminus\Gamma_0),\\
    &u^1(x_1,x_2):=-\frac{2}{\sqrt\pi}\frac{c}{\sqrt{1-c^2}}\partial_1 S\left(\frac{x_1}{\sqrt{1-c^2}},x_2\right)\in L^2(\Omega).
\end{align*}

Let us consider a function $\Psi$ which satisfies the regularity assumptions~\eqref{Psi2} and condition~\eqref{Psi3}, namely
$$\Psi(t)=0\quad\text{on $B_\e(t):=\{x\in \R^2:|x-(ct,0)|<\e\}$ for every $t\in[0,T]$},$$
with $0<\e<R-cT$. In this case $u$ is a weak solution, in the sense of Definition~\ref{def:weaksol}, to the damped wave equation
\begin{equation*}
    \ddot u(t)-\Delta u(t)-\div (\Psi^2(t)\nabla \dot u(t))=f(t)\quad \text{in }\in\Omega\setminus\Gamma_t,\,t\in(0,T),
\end{equation*}
with forcing term $f$ given by
\begin{equation*}
    f:=-\div (\Psi^2\nabla \dot u)=-\nabla\Psi\cdot2\Psi\nabla\dot u-\Psi^2\Delta \dot u\in L^2(0,T;L^2(\Omega)),
\end{equation*}
and boundary and initial conditions
\begin{align*}
&u(t)=w(t) & &\hspace{-2.5cm}\text{on $\partial\Omega$, $t\in(0,T)$},& &\\
& \frac{\partial u}{\partial \nu}(t)+\Psi^2(t) \frac{\partial \dot u}{\partial \nu}(t)=0 & &\hspace{-2.5cm}\text{on $\Gamma_t$,\,\, $t\in(0,T)$}, & & \\
& u(0)=u^0,\quad \dot u(0)=u^1.
\end{align*}
Notice that for the homogeneous Neumann boundary conditions on $\Gamma_t$ we used $\frac{\partial \dot u}{\partial \nu}(t)=\nabla \dot u(t)\cdot \nu=\partial_2\dot u(t)=0$ on $\Gamma_t$. By the uniqueness result proved in the previous section, the function $u$ coincides with that one found in Theorem~\ref{thm:mainresult}. Thanks to the computations done in~\cite[Section 4]{DMLT}, we know that $u$ satisfies for every $t\in[0,T]$ the following energy--dissipation balance for the undamped equation, where $ct$ coincides with the length of $\Gamma_t\setminus\Gamma_0$
\begin{align}\label{eq:en1}
&\frac{1}{2}\norm{\dot u(t)}^2_{L^2(\Omega)}+\frac{1}{2}\norm{\nabla u(t)}^2_{L^2(\Omega;\R^2)}+ct=\frac{1}{2}\norm{\dot u(0)}^2_{L^2(\Omega)}+\frac{1}{2}\norm{\nabla u(0)}^2_{L^2(\Omega;\R^2)}+\int_0^t(\frac{\partial u}{\partial\nu}(s),\dot w(s))_{L^2(\partial\Omega)}\,\de s.
\end{align}
Moreover, we have
\begin{equation}
\begin{aligned}
\int_0^t(\frac{\partial u}{\partial\nu}(s),\dot w(s))_{L^2(\partial\Omega)}\,\de s&=\int_0^t(\nabla u(s),\nabla\dot w(s))_{L^2(\Omega;\R^2)}\,\de s-\int_0^t (\dot u(s),\ddot w(s))_{L^2(\Omega)}\,\de s\\
&\quad+(\dot u(t),\dot w(t))_{L^2(\Omega)}-(\dot u(0),\dot w(0))_{L^2(\Omega)}.
\end{aligned}
\end{equation}
For every $t\in[0,T]$ we compute
\begin{align*}
(f(t),\dot u(t)-\dot w(t))_{L^2(\Omega)}&=-\int_{(\Omega\setminus B_\e(t))\setminus\Gamma_t}\div [\Psi^2(t,x)\nabla \dot u(t,x)](\dot u(t,x)-\dot w(t,x))\,\de x\\
&=-\int_{(\Omega\setminus B_\e(t))\setminus\Gamma_t}\div [\Psi^2(t,x)\nabla \dot u(t,x)(\dot u(t,x)-\dot w(t,x))]\,\de x\\
&\quad+\int_{(\Omega\setminus B_\e(t))\setminus\Gamma_t}\Psi^2(t,x)\nabla \dot u(t,x)\cdot(\nabla\dot u(t,x)-\nabla\dot w(t,x))\,\de x.
\end{align*}
If we denote by $\dot u^\oplus(t)$ and $\dot w^\oplus(t)$ the traces of $\dot u(t)$ and $\dot w(t)$ on $\Gamma_t$ from above and by $\dot u^\ominus(t)$ and $\dot w^\ominus(t)$ the trace from below, thanks to the divergence theorem we have
\begin{align*}
 &\int_{(\Omega\setminus B_\e(t))\setminus\Gamma_t}\div [\Psi^2(t,x)\nabla \dot u(t,x)(\dot u(t,x)-\dot w(t,x))]\,\de x\\
 &=\int_{\partial \Omega} \Psi^2(t,x)\frac{\partial \dot u}{\partial \nu}(t,x)(\dot u(t,x)-\dot w(t,x))\,\de x+\int_{\partial B_\e(t)} \Psi^2(t,x)\frac{\partial \dot u}{\partial \nu}(t,x)(\dot u(t,x)-\dot w(t,x))\,\de x\\
 &\quad-\int_{(\Omega\setminus B_\e(t))\cap\Gamma_t} \Psi^2(t,x)\partial_2 \dot u^\oplus(t,x)(\dot u^\oplus(t,x)-\dot w^\oplus(t,x))\,\de \mathcal H^1(x)\\
 &\quad+\int_{(\Omega\setminus B_\e(t))\cap\Gamma_t} \Psi^2(t,x)\partial_2 \dot u^\ominus(t,x)(\dot u^\ominus(t,x)-\dot w^\ominus(t,x))\,\de \mathcal H^1(x)=0,
\end{align*}
since $u(t)=w(t)$ on $\partial \Omega$, $\Psi(t)=0$ on $\partial B_\e(t)$, and $\partial_2\dot u(t)=0$ on $\Gamma_t$. Therefore for every $t\in[0,T]$ we get
\begin{align}\label{eq:en2}
(f(t),\dot u(t)-\dot w(t))_{L^2(\Omega)}=\|\Psi(t)\nabla \dot u(t)\|^2_{L^2(\Omega;\R^2)}-(\Psi(t)\nabla \dot u(t),\Psi(t)\nabla\dot w(t))_{L^2(\Omega;\R^2)}.
\end{align}
By combining~\eqref{eq:en1}--\eqref{eq:en2} we deduce that $u$ satisfies for every $t\in[0,T]$ the following Griffith's energy--dis\-sipation balance for the viscoelastic dynamic equation
\begin{equation}\label{endisbal}
\begin{aligned}
\frac{1}{2}\norm{\dot u(t)}^2_{L^2(\Omega)}+\frac{1}{2}\norm{\nabla u(t)}^2_{L^2(\Omega;\R^2)}&+\int_0^{t} \norm{\Psi(s) \nabla\dot{u}(s)}^2_{L^2(\Omega;\R^2)}\,\de s+ct\\
&= \frac{1}{2}\norm{\dot u(0)}^2_{L^2(\Omega)}+\frac{1}{2}\norm{\nabla u(0)}^2_{L^2(\Omega;\R^2)}+\mathcal W_{tot}(t),
\end{aligned}
\end{equation}
where in this case the total work takes the form
\begin{align*}
\mathcal W_{tot}(t)&:=\int_0^t \left[(f(s),\dot u(s)-\dot w(s))_{L^2(\Omega)}+(\nabla u(s),\nabla\dot w(s))_{L^2(\Omega;\R^2)}+(\Psi(s)\nabla\dot u(s),\Psi(s)\nabla\dot w(s))_{L^2(\Omega;\R^2)}\right]\de s\\
&\quad-\int_0^t (\dot u(s),\ddot w(s))_{L^2(\Omega)}\,\de s+(\dot u(t),\dot w(t))_{L^2(\Omega)}-(\dot u(0),\dot w(0))_{L^2(\Omega)}.
\end{align*}
Notice that equality~\eqref{endisbal} gives~\eqref{Psi_en_bal}. This show that in this model Griffith's dynamic energy--dissipation balance can be satisfied by a moving crack, in contrast with the case $\Psi=1$, which always leads to~\eqref{en_diss}.

\begin{acknowledgements}
The authors wish to thank Professors Gianni Dal Maso and Rodica Toader for having proposed the problem and for many helpful discussions on the topic. The authors are members of the {\em Gruppo Nazionale per l'Analisi Ma\-te\-ma\-ti\-ca, la Probabilit\`a e le loro Applicazioni} (GNAMPA) of the {\em Istituto Nazionale di Alta Matematica} (INdAM).
\end{acknowledgements}

\end{document}